\documentclass[12pt, reqno]{amsart}
\usepackage{enumerate,amsmath,amssymb,bm,ascmac,amsthm,url}
\usepackage{tikz}
\usepackage[margin=30truemm]{geometry}
\usepackage{here}

\theoremstyle{definition}
\newtheorem{thm}{Theorem}[section]
\newtheorem{thm*}{Theorem}

\newtheorem{defi*}{Definition}
\newtheorem{lem}[thm]{Lemma}
\newtheorem{lem*}{Lemma}
\newtheorem{pro}[thm]{Proposition}
\newtheorem{pro*}{Proposition}

\newtheorem{cor}[thm]{Corollary}

\newcommand{\MC}[1]{\mathcal{#1}}
\newcommand{\MB}[1]{\mathbb{#1}}

\newcommand{\BM}[1]{{\bm #1}}

\newcommand{\marukakko}[1]{\left( #1 \right)}

\newcommand{\G}{\Gamma}

\newcommand{\NH}{\tilde{H}}

\DeclareMathOperator{\Spec}{Spec}

\DeclareMathOperator{\Sym}{Sym}
\DeclareMathOperator{\sgn}{sgn}
\DeclareMathOperator{\diag}{diag}

\makeatletter

\@addtoreset{equation}{section}
\makeatother

\title[
Periodicity of mixed paths and mixed cycles
]{
Periodicity of quantum walks defined by mixed paths and mixed cycles
}

\author[S. Kubota]{Sho Kubota}
\author[H. Sekido]{Hiroto Sekido}
\author[H. Yata]{Harunobu Yata}
\address{Department of Applied Mathematics, Faculty of Engineering, Yokohama National University,
Hodogaya, Yokohama 240-8501, Japan}
\email{kubota-sho-bp@ynu.ac.jp}
\address{Department of Applied Mathematics, Graduate School of Engineering Science,
Yokohama National University, Hodogaya, Yokohama, 240-8501, Japan}
\email{sekido-hiroto-zk@ynu.jp}
\address{Department of Applied Mathematics, Graduate School of Engineering Science,
Yokohama National University, Hodogaya, Yokohama, 240-8501, Japan}
\email{yata-harunobu-dx@ynu.jp}

\date{}

\begin{document}
\maketitle
\begin{abstract}
In this paper,
we determine periodicity of quantum walks defined by mixed paths and mixed cycles.
By the spectral mapping theorem of quantum walks,
consideration of periodicity is reduced to eigenvalue analysis of $\eta$-Hermitian adjacency matrices.
First,
we investigate coefficients of the characteristic polynomials of $\eta$-Hermitian adjacency matrices.
We show that the characteristic polynomials of mixed trees and their underlying graphs are same.
We also define $n+1$ types of mixed cycles
and show that every mixed cycle is switching equivalent to one of them.
We use these results to discuss periodicity.
We show that the mixed paths are periodic for any $\eta$.
In addition,
we provide a necessary and sufficient condition for a mixed cycle to be periodic and determine their periods.

\end{abstract}

\section{Introduction}

Quantum walks are quantum analogues of classical random walks \cite{AAKV, ADZ, Gu}.
In the last two decades,
a great deal of research on quantum walks has been carried out,
and they have strong connections with various fields.
In quantum information,
quantum walk models can be seen as a generalization of Grover's search algorithm \cite{Gr, P}.
An important fact in mathematics is the spectral mapping theorem of quantum walks \cite{HKSS2014, KSY}.
The spectral mapping theorems reduce eigenvalue analysis of time evolution operators
to eigenvalue analysis of other self-adjoint operators.
They bring quantum walks into close connection with functional analysis \cite{SS}
and spectral graph theory \cite{KST}.
In the study of periodicity of discrete-time quantum walk,
field theory and algebraic number theory have also been leveraged \cite{KKKS, SMA}.
Studies of perfect state transfer in continuous-time quantum walks have also been done
by algebraic graph theory and algebraic combinatorics.
We refer to Godsil's survey \cite{G2012}.
Recent studies on state transfer are in \cite{BMW, LW, LLZZ, MDDT, WL, Z}.

\subsection{Related works to periodicity}

The topic discussed in this paper is periodicity of discrete-time quantum walks.
The works of \cite{HKSS2017, KoST} triggered off studies of periodicity on various graphs.
For example,
the studies done on Grover walks can be summarized in Table~\ref{80}.
\begin{table}[h]
  \centering
  \begin{tabular}{|c|c|}
\hline
Graphs & Ref. \\
\hline
\hline
Complete graphs, complete bipartite graphs, SRGs & \cite{HKSS2017} \\ \hline
Generalized Bethe trees & \cite{KSTY2018} \\ \hline
Distance regular graphs & \cite{Y2019} \\ \hline
Cycle (3-state) & \cite{KKKS} \\ \hline
Complete graphs with self loops & \cite{IMT} \\ \hline
\end{tabular}
 \caption{Prior works on periodicity of Grover walks on undirected graphs} \label{80}
\end{table}
In other models,
periodicity of Fourier walks has been considered by Saito \cite{S},
and periodicity of staggered walks has been studied in \cite{KSTY2019}.
Recently, periodicity of quantum walks with generalized Grover coins has been considered
by Sarkar et al~\cite{SMA}.

\subsection{Main Results}
In this paper, we study periodicity of mixed graphs.
There are three main theorems.
See later sections for more detailed terms and definitions.
First, we generalize the result in \cite{AGNN} related to classification of mixed cycles to $\eta$-Hermitian adjacency matrices.
In \cite{AGNN},
Akbari et al~provided several typical switching functions and classified mixed cycles into three types.
We give similar considerations in $\eta$-Hermitian adjacency matrices.
Among the four types of switching defined in \cite{AGNN},
one switching cannot be used in $\eta$-Hermitian adjacency matrices.
Due to this,
we show that there are at most $n+1$ switching equivalence classes of mixed cycles
in $\eta$-Hermitian adjacency matrices.
The claim is as follows:

\begin{thm} \label{Main1}
{\it
Let $G = (V, \MC{A})$ be a mixed cycle of length $n$.
Then, there exists $j \in \{0,1,\dots, n\}$ such that $G$ and $C_n^j$ is $H_{\eta}$-cospectral.
Moreover, we have
\[ \det H_{\eta}(C_n^j) = (-1)^{n+1} 2 \cos (\eta j) + (-1)^{\lfloor \frac{n}{2} \rfloor} (1 + (-1)^n). \]
}
\end{thm}

The second main theorem is to determine periodicity of mixed paths.
Using the model defined in \cite{KST},
we study periodicity of quantum walks defined by mixed graphs.
This model is defined by both a mixed graph and a real number $\eta \in [0, 2\pi)$.
The claim is as follows:

\begin{thm} \label{Main2}
{\it
Let $G = (V, \MC{A})$ be a mixed path on $n$ vertices equipped with an $\eta$-function $\theta$.
Then $G$ is periodic for any $\eta \in [0, 2\pi)$, and the period is $2(n-1)$.
}
\end{thm}

The third main theorem is to determine periodicity of mixed cycles.
Since periodicity is determined by eigenvalues,
it is sufficient to consider the only $n+1$ types of mixed cycles by the first main theorem.
The claim is as follows:

\begin{thm} \label{Main3}
{\it
Let $G = (V, \MC{A})$ be a mixed cycle on $n$ vertices equipped with an $\eta$-function $\theta$.
Then,
$G$ is periodic if and only if $\eta \in \MB{Q}\pi$.
In addition,
we suppose that $\eta \in \MB{Q}\pi$ and the mixed cycle $G$ is $H_{\eta}$-cospectral with $C_n^j$.
Let $\eta = \frac{p}{q}\pi$, where $p$ and $q$ are coprime.
Then, the period $\tau$ of $G$ is the following:
\begin{equation} \label{65m}
\tau =
\begin{cases}
\frac{2qn}{(j, 2q)} \quad &\text{if $p$ is odd,} \\
\frac{qn}{(j, q)} \quad &\text{if $p$ is even.} \\
\end{cases}
\end{equation}
}
\end{thm}

This paper is organized as follows.
In Section~\ref{102},
we prepare terminologies on spectral graph theory.
The definitions of mixed graphs and their $\eta$-Hermitian adjacency matrices are provided.
In Section~\ref{11},
coefficients of the characteristic polynomials of $\eta$-Hermitian adjacency matrices are discussed.
We focus on permutations that contribute to values of determinants.
Relationship between the characteristic polynomials of a mixed graph and its underlying graph is clarified.
In Section~\ref{104},
we carry out classification of mixed cycles by $\eta$-Hermitian adjacency matrices.
We introduce $n+1$ types of mixed cycles
and show that every mixed cycle is switching equivalent to one of them.
On the other hand, we show that the $n+1$ types of mixed cycles have different eigenvalues except for a finite number of $\eta$.
In Section~\ref{105},
we prepare quantum walks defined by mixed graphs.
We define periodicity of mixed graphs and provide some characterizations of it.
In Section~\ref{106},
we discuss periodicity of mixed paths.
We observe action of time evolution matrices on the unit vectors and provide a visual proof.
In Section~\ref{107},
we discuss periodicity of mixed cycles.
It is easy to provide a necessary and sufficient condition for a mixed cycle to be periodic,
but determination of the period is a bit complicated.

\section{Preliminaries on spectral graph theory} \label{102}

Let $\G =(V, E)$ be a finite simple and connected graph with the vertex set $V$ and the edge set $E$.
For $x \in V$,
the set of neighbors of $x$ is denoted by $N(x)$.
Define $\MC{A} = \MC{A}(\G)=\{ (x, y), (y, x) \mid xy \in E(\G) \}$,
which is the set of the symmetric arcs of $\G$.
The origin and terminus of $a=(x, y) \in \MC{A}$ are denoted by $o(a), t(a)$, respectively.
We express $a^{-1}$ as the inverse arc of $a$.

A {\it mixed graph} $G$ consists of a finite set $V(G)$ of vertices
together with a subset $\MC{A}(G) \subset V(G) \times V(G) \setminus \{ (x,x) \mid x \in V \}$
of ordered pairs called {\it arcs}.
Let $G$ be a mixed graph.
Define $\MC{A}^{-1}(G) = \{ a^{-1} \mid a \in \MC{A}(G) \}$
and $\MC{A}^{\pm}(G) = \MC{A}(G) \cup \MC{A}^{-1}(G)$.
If there is no danger of confusion, we write $\MC{A}(G)$ as $\MC{A}$ simply.
If $(x,y) \in \MC{A} \cap \MC{A}^{-1}$,
we say that the unordered pair $\{x,y\}$ is a {\it digon} of $G$.
For a vertex $x \in V(G)$,
define $\deg_{G} x = \deg_{G^{\pm}} x$.
A mixed graph $G$ is $k$-regular if $\deg_{G}x = k$ for any vertex $x \in V(G)$.
The graph $G^{\pm} = (V(G), \MC{A}^{\pm})$ is
so-called the {\it underlying graph} of a mixed graph $G$,
and this is regarded as an undirected graph depending on context.
On the other hand,
we equate an undirected graph with a mixed graph
by considering undirected edges $xy$ as bidirectional arcs $(x,y), (y,x)$.
Throughout this paper, we assume that mixed graphs are weakly connected,
i.e.,
we assume that $G^{\pm}$ is connected.


Let $G = (V, \MC{A})$ be a mixed graph.
For $\eta \in [0, 2\pi)$, 
the {\it $\eta$-Hermitian adjacency matrix} $H_{\eta} = H_{\eta}(G) \in \MB{C}^{V \times V}$ is defined by
\[
(H_{\eta})_{x,y} = \begin{cases}
1 \qquad &\text{if $(x,y) \in \MC{A} \cap \MC{A}^{-1}$,} \\
e^{\eta i} \qquad &\text{if $(x,y) \in \MC{A} \setminus \MC{A}^{-1}$,} \\
e^{-\eta i} \qquad &\text{if $(x,y) \in \MC{A}^{-1} \setminus \MC{A}$,} \\
0 \qquad &\text{otherwise.}
\end{cases}
\]
When $\eta = \frac{\pi}{2}$,
the matrix $H_{\frac{\pi}{2}}$ is nothing but the Hermitian adjacency matrix.
This is introduced by Guo--Mohar \cite{GM} and Li--Liu \cite{LL}, independently.
When $\eta = \frac{\pi}{3}$,
the matrix $H_{\frac{\pi}{3}}$ is called the Hermitian adjacency matrix of the second kind.
This is introduced by Mohar \cite{M}.
We refer to \cite{AAS, GS, LY} as recent studies on Hermitian adjacency matrices.
Note that $H_{\eta}(G^{\pm})$ coincides with the ordinary adjacency matrix of $G^{\pm}$.
Define the {\it degree matrix} $D = D(G) \in \MB{C}^{V \times V}$ by $D_{x,y} = (\deg_{G}x)\delta_{x,y}$
for vertices $x,y \in V(G)$,
where $\delta_{x,y}$ is the Kronecker delta.
For $\eta \in [0, 2\pi)$,
the {\it normalized $\eta$-Hermitian adjacency matrix} $\NH_{\eta}$ is defined by
\[ \NH_{\eta} = D^{-\frac12} H_{\eta} D^{-\frac12}. \]
Note that if a mixed graph $G$ is $k$-regular, we have $\NH_{\eta} = \frac{1}{k} H_{\eta}$.

Let $G$ be a mixed graph.
The list of the eigenvalues of $H_{\eta}(G)$ together with their multiplicities,
denoted by $\Spec(H_{\eta}(G))$, is called {\it $H_{\eta}$-spectrum} of $G$.
We say that mixed graphs $G$ and $G'$ are {\it $H_{\eta}$-cospectral} if they have the same $H_{\eta}$-spectrum.
The same is on $\tilde{H}_{\eta}$.

\section{Permutations and characteristic polynomials} \label{11}
Let $\G$ be an undirected graph.
A mixed graph $G$ is said to be a {\it mixed $\G$} if $G^{\pm}$ is isomorphic to $\G$.
Similarly, we say that $G$ is a {\it mixed tree} if $G^{\pm}$ is a tree.
Let $G = (V, \MC{A})$ be a mixed graph, and let $x_1,x_2, \dots, x_l \in V$.
We say that a sequence $C = (x_1,x_2, \dots, x_l)$ is an {\it $l$-cycle} in $G$
if $C$ is an $l$-cycle in $G^{\pm}$.
The {\it girth} of $G$ is defined by the girth of $ G^{\pm}$.
Note that if the girth of $G$ is $s+1$, then it has no $l$-cycle for $l \in \{1,2, \dots, s\}$.

Let $G = (V,\MC{A})$ be a mixed graph.
Put $X = \lambda I - H_{\eta}(G)$.
Then
\[ \det X = \sum_{\sigma \in \Sym(V)} \sgn(\sigma) \prod_{x \in V} X_{x, \sigma(x)}, \]
where $\Sym(V)$ is the set of all permutations of $V$.
Let
\[
\MC{P}(X) := \left\{ \sigma \in \Sym(V) \, \middle | \, \prod_{x \in V} X_{x, \sigma(x)} \neq 0 \right\}.
\]
In addition, we define
\[
\MC{P}_m(X) := \Big\{ \sigma \in \MC{P}(X) \, \Big | \, |\{ x \in V \mid \sigma(x) \neq x \} | = m \Big\}
\]
for $m \in \MB{N}$.
Any permutation $\sigma \in \Sym(V)$ can be expressed as the product of disjoint cyclic permutations,
say $\sigma = \sigma_1 \sigma_2 \cdots \sigma_l$ for some $l \in \MB{N}$.
We call each $\sigma_i$ a {\it factor} of $\sigma$.

\begin{lem} \label{00}
{\it 
Let $G = (V, \MC{A})$ be a mixed graph with girth $s+1$,
and let $X = \lambda I - H_{\eta}(G)$.
For $l \in \{1,2, \dots, s\}$,
all factors of $\sigma \in \MC{P}_{l}(X)$ are transpositions,
where $n = |V|$.
}
\end{lem}
\begin{proof}
Let $H_{\eta} = H_{\eta}(G)$.
Suppose $\sigma \in \MC{P}_{l}(X)$ has a factor $\sigma_i$ of length $t > 2$.
Display as $\sigma_i = (x_1x_2 \cdots x_t)$.
Then $(H_{\eta})_{x_1, x_2} (H_{\eta})_{x_2, x_3} \cdots (H_{\eta})_{x_t, x_1} \neq 0$.
However, $G$ has no $t$-cycles since $t \leq l \leq s$.
Thus, at least one of $(H_{\eta})_{x_1, x_2}, (H_{\eta})_{x_2, x_3}, \dots, (H_{\eta})_{x_t, x_1}$ is $0$.
This is a contradiction.
\end{proof}

On the other hand,
for distinct vertices $x,y$ in a mixed graph $G$,
we have
\begin{equation} \label{01}
(H_{\eta})_{x,y} (H_{\eta})_{y,x} = |(H_{\eta})_{x,y}|^2 =
\begin{cases}
1 \qquad &\text{if $x$ is adjacent to $y$ in $G^{\pm}$,} \\
0 \qquad &\text{otherwise,}
\end{cases}
\end{equation}
since $H_{\eta}$ is Hermitian.
Remarking this, we have the following.

\begin{pro} \label{10}
{\it
Let $G = (V, \MC{A})$ be a mixed graph with girth $s+1$.
Let
\begin{align*}
\det (\lambda I - H_{\eta}(G)) &= \lambda^n + a_1 \lambda^{n-1} + \cdots + a_{n-1} \lambda + a_n, \\
\det (\lambda I - H_{\eta}(G^{\pm})) &= \lambda^n + b_1 \lambda^{n-1} + \cdots + b_{n-1} \lambda + b_n,
\end{align*}
where $n = |V|$.
Then we have $a_l = b_l$ for any $l \in \{1,2, \dots, s\}$.
}
\end{pro}

\begin{proof}
Put $X = \lambda I - H_{\eta}(G)$ and $Y = \lambda I - H_{\eta}(G^{\pm})$.
For $x, y \in V$, $X_{x,y} \neq 0$ if and only if $Y_{x,y} \neq 0$,
so $\MC{P}(X) = \MC{P}(Y)$.
In particular, $\MC{P}_{l}(X) = \MC{P}_{l}(Y)$.
If $\MC{P}_{l}(X) = \emptyset$, we have $a_l = b_l = 0$.
We consider $\MC{P}_{l}(X) \neq \emptyset$.
Let $\sigma \in \MC{P}_{l}(X)$.
By Lemma~\ref{00},
all factors of $\sigma$ are transpositions.
Display as $\sigma = (x_1 y_1)(x_2 y_2) \cdots (x_t y_t)$,
where $t = l / 2$.
By (\ref{01}),
\begin{align*}
\sgn(\sigma) \prod_{x \in V} X_{x, \sigma(x)}
&= \lambda^{n-2t} (H_{\eta}(G))_{x_1, y_1} (H_{\eta}(G))_{y_1, x_1}  \cdots (H_{\eta}(G))_{x_t, y_t} (H_{\eta}(G))_{y_t, x_t} \\
&= \lambda^{n-2t} \cdot 1 \\
&= \lambda^{n-2t} (H_{\eta}(G^{\pm}))_{x_1, y_1} (H_{\eta}(G^{\pm}))_{y_1, x_1}  \cdots (H_{\eta}(G^{\pm}))_{x_t, y_t} (H_{\eta}(G^{\pm}))_{y_t, x_t} \\
&= \sgn(\sigma) \prod_{x \in V} Y_{x, \sigma(x)}.
\end{align*}
Therefore,
\[ a_l 
= \sum_{\sigma \in \MC{P}_{l}(X)} \sgn(\sigma) \prod_{x \in V} X_{x, \sigma(x)}
= \sum_{\sigma \in \MC{P}_{l}(Y)} \sgn(\sigma) \prod_{x \in V} Y_{x, \sigma(x)}
= b_l. \]
\end{proof}

\begin{cor} \label{60}
{\it
Let $G$ be a mixed tree.
Then
\[ \det (\lambda I - H_{\eta}(G)) = \det (\lambda I - H_{\eta}(G^{\pm})), \]
i.e.,
$G$ and $G^{\pm}$ are $H_{\eta}$-cospectral.
}
\end{cor}

\begin{proof}
Since $G$ is a mixed tree, the girth is $\infty$.
By Proposition~\ref{10},
all coefficients of both characteristic polynomials are equal.
\end{proof}

For distinct vertices $x,y$ in a mixed graph $G$,
we also have
\[
(\NH_{\eta})_{x,y} (\NH_{\eta})_{y,x} = |(\NH_{\eta})_{x,y}|^2 =
\begin{cases}
\frac{1}{\deg x \deg y} \qquad &\text{if $x$ is adjacent to $y$ in $G^{\pm}$,} \\
0 \qquad &\text{otherwise.}
\end{cases}
\]
Therefore, the same result as Proposition~\ref{10} holds for $\NH_{\eta}$.

\begin{pro} \label{12}
{\it
Let $G = (V, \MC{A})$ be a mixed graph with girth $s+1$.
Let
\begin{align*}
\det (\lambda I - \NH_{\eta}(G)) &= \lambda^n + a_1 \lambda^{n-1} + \cdots + a_{n-1} \lambda + a_n, \\
\det (\lambda I - \NH_{\eta}(G^{\pm})) &= \lambda^n + b_1 \lambda^{n-1} + \cdots + b_{n-1} \lambda + b_n,
\end{align*}
where $n = |V|$.
Then we have $a_l = b_l$ for any $l \in \{1,2, \dots, s\}$.
In particular,
if $G$ is a mixed tree, then
\[ \det (\lambda I - \NH_{\eta}(G)) = \det (\lambda I - \NH_{\eta}(G^{\pm})), \]
i.e.,
$G$ and $G^{\pm}$ are $\NH_{\eta}$-cospectral.
}
\end{pro}

\section{Classification of mixed cycles by $H_{\eta}$-spectra} \label{104}

The relation that two mixed graphs are $H_{\eta}$-cospectral is an equivalence relation.
We call its equivalence class {\it $H_{\eta}$-cospectral class}.
In this section,
we determine the equivalence classes in the mixed cycles.

Let $G = (V, \MC{A})$ be a mixed graph.
A function $\alpha : V \to \{ 1, e^{\pm i \eta} \}$ is called a {\it switching function}.
For a switching function $\alpha$, 
we define the matrix $D(\alpha) \in \MB{C}^{V \times V}$ by $D(\alpha)_{x,y} = \alpha(x) \delta_{x,y}$.
Taking a switching function $\alpha$ well,
the matrix $D(\alpha) H_{\eta}(G) D(\alpha)^*$ is the $\eta$-Hermitian adjacency matrix of another mixed graph $G'$.
Then, we say that {\it $G'$ is obtained by switching with respect to $\alpha$ from $G$}.
Clearly,
$G$ and $G'$ are $H_{\eta}$-cospectral.
Note that if a mixed graph $G'$ is obtained by switching with respect to $\alpha$ from a mixed graph $G$,
we also say that $G$ and $G'$ are {\it switching equivalent}.
Recent studies related to switching equivalence of mixed graphs are in \cite{KB, WY}.

\subsection{Switching functions}
In \cite{AGNN},
Akbari et al~defined the four typical switching functions
and determined the $H_{\frac{\pi}{2}}$-cospectral classes in the mixed cycles.
We generalize their result to general $\eta \in [0, 2\pi)$.
Let $G = (V, \MC{A})$ be a mixed cycle.
First,
we define the three typical switching functions as follows:

\begin{itemize}
\item[Sw.2$'$.] For a vertex $x \in V$, define the switching function
\[ \alpha(v) = \begin{cases}
e^{i \eta} \qquad &\text{if $v = x$,} \\
1 \qquad &\text{otherwise}.
\end{cases} \]
Let $N_{G^{\pm}}(x) = \{v_1, v_2\}$.
If $(v_1, x), (v_2, x) \in \MC{A} \setminus \MC{A}^{-1}$,
then we have the mixed graph $(V, \MC{A} \cup \{ (v_1, x)^{-1}, (v_2, x)^{-1} \})$ by switching.

\begin{figure}[h]
\begin{center}
\begin{tikzpicture}
[scale = 0.5,
fat/.style={circle,fill=black, inner sep = 2mm},
slim/.style={circle,fill=black, inner sep = 0.8mm},
]
\node[slim] (s1) at (-2,2) [label = below:$v_1$] {};
\node[slim] (s2) at (0,3) [label = above:$x$] {};
\node[slim] (s3) at (2,2) [label = below:$v_2$] {};
\draw[line width = 1pt][->] (s1) to  (s2);
\draw[line width = 1pt][->] (s3) to  (s2);
\end{tikzpicture} \raisebox{10mm}{$\xrightarrow{\text{Sw.2$'$ on $x$}}$}
\begin{tikzpicture}
[scale = 0.5,
fat/.style={circle,fill=black, inner sep = 2mm},
slim/.style={circle,fill=black, inner sep = 0.8mm},
]
\node[slim] (s1) at (-2,2) [label = below:$v_1$] {};
\node[slim] (s2) at (0,3) [label = above:$x$] {};
\node[slim] (s3) at (2,2) [label = below:$v_2$] {};
\draw[line width = 1pt] (s1) -- (s2) -- (s3);
\end{tikzpicture}
\caption{Sw.2$'$ on the vertex $x$} \label{sw2}
\end{center}
\end{figure}
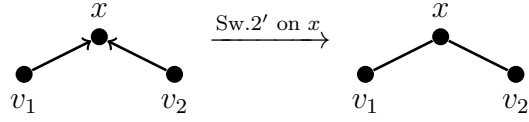

\item[Sw.3$'$.] For a vertex $x \in V$, define the switching function
\[ \alpha(v) = \begin{cases}
e^{-i \eta} \qquad &\text{if $v = x$,} \\
1 \qquad &\text{otherwise}.
\end{cases} \]
Let $N_{G^{\pm}}(x) = \{v_1, v_2\}$.
If $(x, v_1), (x, v_2) \in \MC{A} \setminus \MC{A}^{-1}$,
then we have the mixed graph $(V, \MC{A} \cup \{ (x, v_1)^{-1}, (x, v_2)^{-1} \})$ by switching.

\begin{figure}[h]
\begin{center}
\begin{tikzpicture}
[scale = 0.5,
fat/.style={circle,fill=black, inner sep = 2mm},
slim/.style={circle,fill=black, inner sep = 0.8mm},
]
\node[slim] (s1) at (-2,1) [label = below:$v_1$] {};
\node[slim] (s2) at (0,2) [label = above:$x$] {};
\node[slim] (s3) at (2,1) [label = below:$v_2$] {};
\draw[line width = 1pt][->] (s2) to  (s1);
\draw[line width = 1pt][->] (s2) to  (s3);
\end{tikzpicture} \raisebox{10mm}{$\xrightarrow{\text{Sw.3$'$ on $x$}}$}
\begin{tikzpicture}
[scale = 0.5,
fat/.style={circle,fill=black, inner sep = 2mm},
slim/.style={circle,fill=black, inner sep = 0.8mm},
]
\node[slim] (s1) at (-2,1) [label = below:$v_1$] {};
\node[slim] (s2) at (0,2) [label = above:$x$] {};
\node[slim] (s3) at (2,1) [label = below:$v_2$] {};
\draw[line width = 1pt] (s1) -- (s2) -- (s3);
\end{tikzpicture}
\caption{Sw.3$'$ on the vertex $x$} \label{sw3}
\end{center}
\end{figure}
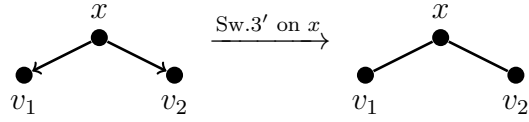

\item[Sw.4$'$.] For a vertex $x \in V$, define the switching function
\[ \alpha(v) = \begin{cases}
e^{i \eta} \qquad &\text{if $v = x$,} \\
1 \qquad &\text{otherwise}.
\end{cases} \]
Let $N_{G^{\pm}}(x) = \{v_1, v_2\}$.
If $(v_1, x) \in \MC{A} \setminus \MC{A}^{-1}$ and $(x, v_2) \in \MC{A} \cap \MC{A}^{-1}$,
then we have the mixed graph $(V, (\MC{A} \setminus \{(v_2, x)\}) \cup \{ (v_1, x)^{-1} \})$ by switching.

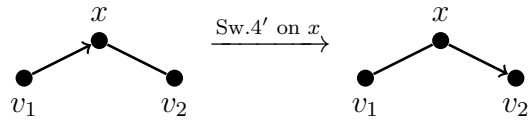
\begin{figure}[h]
\begin{center}
\begin{tikzpicture}
[scale = 0.5,
fat/.style={circle,fill=black, inner sep = 2mm},
slim/.style={circle,fill=black, inner sep = 0.8mm},
]
\node[slim] (s1) at (-2,1) [label = below:$v_1$] {};
\node[slim] (s2) at (0,2) [label = above:$x$] {};
\node[slim] (s3) at (2,1) [label = below:$v_2$] {};
\draw[line width = 1pt][->] (s1) to  (s2);
\draw[line width = 1pt] (s2) -- (s3);
\end{tikzpicture} \raisebox{10mm}{$\xrightarrow{\text{Sw.4$'$ on $x$}}$}
\begin{tikzpicture}
[scale = 0.5,
fat/.style={circle,fill=black, inner sep = 2mm},
slim/.style={circle,fill=black, inner sep = 0.8mm},
]
\node[slim] (s1) at (-2,1) [label = below:$v_1$] {};
\node[slim] (s2) at (0,2) [label = above:$x$] {};
\node[slim] (s3) at (2,1) [label = below:$v_2$] {};
\draw[line width = 1pt] (s1) -- (s2);
\draw[line width = 1pt][->] (s2) to  (s3);
\end{tikzpicture}
\caption{Sw.4$'$ on the vertex $x$} \label{sw4}
\end{center}
\end{figure}
\end{itemize}

See also Figures~\ref{sw2}, \ref{sw3} and \ref{sw4}.
The above switching functions are named after \cite{AGNN}.
The lack of ``Sw.1$'$" is due to the generalization of $\eta$.

\subsection{The $n+1$ types of mixed cycles}
Let $n \in \MB{N}$ and let $j \in \{0,1, \dots, n\}$.
We define the {\it mixed cycle $C_n^{j}$ of type $j$} by
\begin{align*}
V(C_n^{j}) &= \{ x_1, x_2, \dots, x_n \}, \\
\MC{A}(C_n^{j}) &= \{ (x_1, x_2), \dots, (x_j, x_{j+1}) \} \cup \{ (x_{j+1}, x_{j+2}), \dots, (x_n, x_{n+1}) \}^{\pm},
\end{align*}
where we set $x_{n+1} = x_1$.
We provide the mixed cycles $C_8^3$ and $C_8^8$ in Figure~\ref{31a} as examples.

\begin{figure}[h]
\begin{center}
\begin{tikzpicture}
[scale = 0.5,
slim/.style={circle,fill=black, inner sep = 0.8mm},
]
\node[slim] (v1) at (0,2) [label = above:$x_1$] {};
\node[slim] (v2) at (1.414,1.414) [label =above right:$x_2$] {};
\node[slim] (v3) at (2,0) [label = right:$x_3$] {};
\node[slim] (v4) at (1.414,-1.414) [label =below right:$x_4$] {};
\node[slim] (v5) at (0,-2) [label = below:$x_5$] {};
\node[slim] (v6) at (-1.414,-1.414) [label = below left:$x_6$] {};
\node[slim] (v7) at (-2,0) [label = left:$x_7$] {};
\node[slim] (v8) at (-1.414,1.414) [label = above left:$x_8$] {};
\draw[line width = 1pt][->] (v2) to  (v3);
\draw[line width = 1pt] (v4) to  (v5);
\draw[line width = 1pt] (v6) to  (v5);
\draw[line width = 1pt] (v7) to  (v6);
\draw[line width = 1pt] (v7) to  (v8);
\draw[line width = 1pt] (v8) to  (v1);
\draw[line width = 1pt][->] (v3) to  (v4);
\draw[line width = 1pt][->] (v1) to (v2);
\end{tikzpicture}
$\qquad$
\begin{tikzpicture}
[scale = 0.5,
slim/.style={circle,fill=black, inner sep = 0.8mm},
]
\node[slim] (v1) at (0,2) [label = above:$x_1$] {};
\node[slim] (v2) at (1.414,1.414) [label =above right:$x_2$] {};
\node[slim] (v3) at (2,0) [label = right:$x_3$] {};
\node[slim] (v4) at (1.414,-1.414) [label =below right:$x_4$] {};
\node[slim] (v5) at (0,-2) [label = below:$x_5$] {};
\node[slim] (v6) at (-1.414,-1.414) [label = below left:$x_6$] {};
\node[slim] (v7) at (-2,0) [label = left:$x_7$] {};
\node[slim] (v8) at (-1.414,1.414) [label = above left:$x_8$] {};
\draw[line width = 1pt][->] (v2) to  (v3);
\draw[line width = 1pt][->] (v4) to  (v5);
\draw[line width = 1pt][->] (v5) to  (v6);
\draw[line width = 1pt][->] (v6) to  (v7);
\draw[line width = 1pt][->] (v7) to  (v8);
\draw[line width = 1pt][->] (v8) to  (v1);
\draw[line width = 1pt][->] (v3) to  (v4);
\draw[line width = 1pt][->] (v1) to (v2);
\end{tikzpicture}
\caption{The mixed cycles $C_8^3$ and $C_8^8$} \label{31a}
\end{center}
\end{figure}
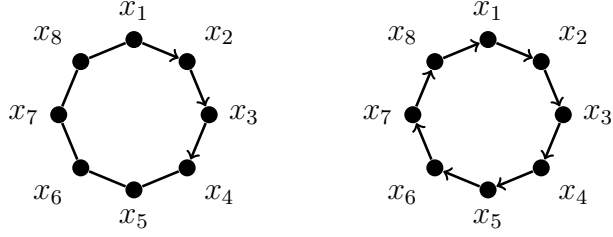

Note that $C_n^0$ is the undirected cycle of length $n$.
We will show that any mixed cycle is $H_{\eta}$-cospectral with some type of the mixed cycle,
and different types of mixed cycles have different spectra except for a finite number of $\eta$.

\begin{lem} \label{02}
{\it
Let $G = (V, \MC{A})$ be a mixed cycle of length $n$,
and let $x, v_1, v_2, \dots, v_l, y \in V$.
If $(x,v_1) \in \MC{A} \setminus \MC{A}^{-1}$,
$(v_1, v_2), (v_2, v_3), \dots, (v_{l-1}, v_l) \in \MC{A} \cap \MC{A}^{-1}$,
and $(v_l, y) \in \MC{A}^{-1} \setminus \MC{A}$.
Then, the mixed cycle $(V, \MC{A} \cup \{ (x,v_1)^{-1}, (v_l, y) \} )$ is obtained by switching from $G$.
}
\end{lem}

\begin{proof}
We prove by induction on $l$.
Consider $l = 1$.
Applying Sw.2$'$ with respect to $v_1$, we have the statement.
We suppose that the statement follows in the case of $l - 1$.
We apply Sw.4$'$ with respect to $v_1$.
The switched mixed cycle is in the situation of the case of $l-1$.
By the assumption of the induction, we have the statement.
\end{proof}

If $C^j_n=(V, \MC{A})$, then we let $(C^j_n )^{-1} =(V, \MC{A}^{-1} )$. 

\begin{pro} \label{03}
{\it
Let $G = (V, \MC{A})$ be a mixed cycle of length $n$.
Then, there exists $j \in \{0,1,\dots, n\}$ such that $G$ and $C_n^j$ is $H_{\eta}$-cospectral.
}
\end{pro}

\begin{proof}
By Lemma~\ref{02},
we have a switched mixed cycle $G'$ such that
the directions of all arcs are aligned clockwise or anticlockwise.
Applying Sw.4$'$ many times, arcs in the graph are replaced consecutively.
This graph is $C_n^j$ or $(C_n^k)^{-1}$ for some $j, k \in \{0,1,\dots, n\}$.
Note that $(C_n^k)^{-1}$ is isomorphic to $C_n^k$.
\end{proof}

Figure~\ref{30} shows that switching yields the mixed graph of type $3$ from a mixed cycle.
Note that Sw.3$'$ is actually unnecessary.
However, it can be used to obtain $C_n^j$ with less switching in some cases.

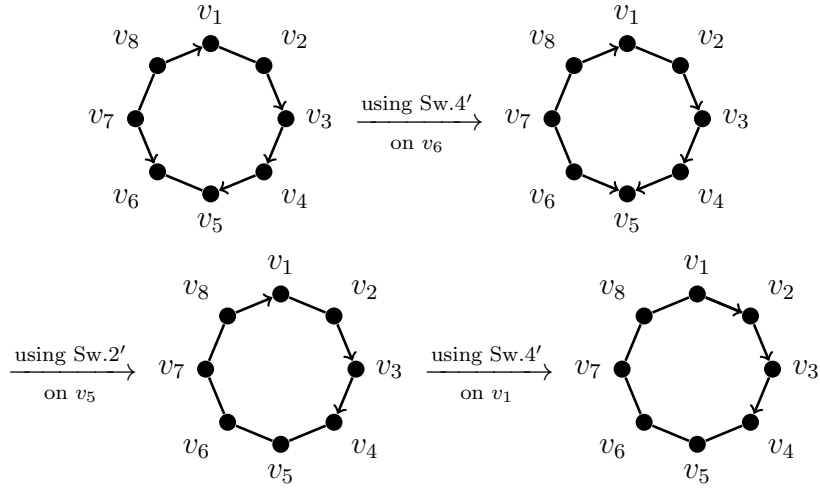
\begin{figure}[h]
\begin{center}
{
\begin{tikzpicture}
[scale = 0.5,
slim/.style={circle,fill=black, inner sep = 0.8mm},
]
\node[slim] (v1) at (0,2) [label = above:$v_1$] {};
\node[slim] (v2) at (1.414,1.414) [label =above right:$v_2$] {};
\node[slim] (v3) at (2,0) [label = right:$v_3$] {};
\node[slim] (v4) at (1.414,-1.414) [label =below right:$v_4$] {};
\node[slim] (v5) at (0,-2) [label = below:$v_5$] {};
\node[slim] (v6) at (-1.414,-1.414) [label = below left:$v_6$] {};
\node[slim] (v7) at (-2,0) [label = left:$v_7$] {};
\node[slim] (v8) at (-1.414,1.414) [label = above left:$v_8$] {};
\draw[line width = 1pt][->] (v2) to  (v3);
\draw[line width = 1pt][->] (v4) to  (v5);
\draw[line width = 1pt] (v6) to  (v5);
\draw[line width = 1pt][->] (v7) to  (v6);
\draw[line width = 1pt] (v7) to  (v8);
\draw[line width = 1pt][->] (v8) to  (v1);
\draw[line width = 1pt][->] (v3) to  (v4);
\draw[line width = 1pt] (v1) -- (v2);
\end{tikzpicture}} \raisebox{15mm}{$\xrightarrow[\text{on $v_6$}]{\text{using Sw.4$'$}}$}
{
\begin{tikzpicture}
[scale = 0.5,
slim/.style={circle,fill=black, inner sep = 0.8mm},
]
\node[slim] (v1) at (0,2) [label = above:$v_1$] {};
\node[slim] (v2) at (1.414,1.414) [label =above right:$v_2$] {};
\node[slim] (v3) at (2,0) [label = right:$v_3$] {};
\node[slim] (v4) at (1.414,-1.414) [label =below right:$v_4$] {};
\node[slim] (v5) at (0,-2) [label = below:$v_5$] {};
\node[slim] (v6) at (-1.414,-1.414) [label = below left:$v_6$] {};
\node[slim] (v7) at (-2,0) [label = left:$v_7$] {};
\node[slim] (v8) at (-1.414,1.414) [label = above left:$v_8$] {};
\draw[line width = 1pt][->] (v2) to  (v3);
\draw[line width = 1pt][->] (v3) to  (v4);
\draw[line width = 1pt] (v7) to  (v6);
\draw[line width = 1pt] (v7) to  (v8);
\draw[line width = 1pt][->] (v8) to  (v1);
\draw[line width = 1pt] (v1) -- (v2);
\draw[line width = 1pt][->] (v4) to (v5);
\draw[line width = 1pt][->] (v6) to (v5);
\end{tikzpicture}}\\
 \raisebox{15mm}{$\xrightarrow[\text{on $v_5$}]{\text{using Sw.2$'$}}$}
{
\begin{tikzpicture}
[scale = 0.5,
slim/.style={circle,fill=black, inner sep = 0.8mm},
]
\node[slim] (v1) at (0,2) [label = above:$v_1$] {};
\node[slim] (v2) at (1.414,1.414) [label =above right:$v_2$] {};
\node[slim] (v3) at (2,0) [label = right:$v_3$] {};
\node[slim] (v4) at (1.414,-1.414) [label =below right:$v_4$] {};
\node[slim] (v5) at (0,-2) [label = below:$v_5$] {};
\node[slim] (v6) at (-1.414,-1.414) [label = below left:$v_6$] {};
\node[slim] (v7) at (-2,0) [label = left:$v_7$] {};
\node[slim] (v8) at (-1.414,1.414) [label = above left:$v_8$] {};
\draw[line width = 1pt][->] (v2) to  (v3);
\draw[line width = 1pt][->] (v3) to  (v4);
\draw[line width = 1pt][->] (v8) to  (v1);
\draw[line width = 1pt] (v1) -- (v2);
\draw[line width = 1pt] (v4) -- (v5) -- (v6) -- (v7) -- (v8);
\end{tikzpicture}}
\raisebox{15mm}{$\xrightarrow[\text{on $v_1$}]{\text{using Sw.4$'$}}$}
{
\begin{tikzpicture}
[scale = 0.5,
slim/.style={circle,fill=black, inner sep = 0.8mm},
]
\node[slim] (v1) at (0,2) [label = above:$v_1$] {};
\node[slim] (v2) at (1.414,1.414) [label =above right:$v_2$] {};
\node[slim] (v3) at (2,0) [label = right:$v_3$] {};
\node[slim] (v4) at (1.414,-1.414) [label =below right:$v_4$] {};
\node[slim] (v5) at (0,-2) [label = below:$v_5$] {};
\node[slim] (v6) at (-1.414,-1.414) [label = below left:$v_6$] {};
\node[slim] (v7) at (-2,0) [label = left:$v_7$] {};
\node[slim] (v8) at (-1.414,1.414) [label = above left:$v_8$] {};
\draw[line width = 1pt][->] (v2) to  (v3);
\draw[line width = 1pt][->] (v3) to  (v4);
\draw[line width = 1pt][->] (v1) to  (v2);
\draw[line width = 1pt] (v1) -- (v2);
\draw[line width = 1pt] (v4) -- (v5) -- (v6) -- (v7) -- (v8) -- (v1);
\end{tikzpicture}}
\caption{Switching a mixed cycle} \label{30}
\end{center}
\end{figure}

\subsection{Determinants of $H_{\eta}(C_n^j)$}
Next, we show that different types of mixed cycles have different spectra.
We focus on the constant term of the characteristic polynomial of $H_{\eta}(C_n^j)$.
Let $P_n$ be the undirected path graph on $n$ vertices.

\begin{lem} \label{50}
{\it
We have
\[ \det H_{\eta}(P_n) = (-1)^{\lfloor \frac{n}{2} \rfloor} \frac{1+(-1)^n}{2}. \]
}
\end{lem}

\begin{proof}
We will show that
\[ \det H_{\eta}(P_n) = \begin{cases}
1 \qquad &\text{if $n \equiv 0 \pmod 4$,} \\
-1 \qquad &\text{if $n \equiv 2 \pmod 4$,} \\
0 \qquad &\text{if $n \equiv 1,3 \pmod 4$.} \\
\end{cases} \]
The determinant is
\[
\det(H_{\eta}(P_{n})) =
\begin{vmatrix}
0 & 1 & 0 & \cdots & \cdots & \cdots & 0 \\
1 & 0 & 1 & 0 & \cdots & \cdots & 0 \\
0 & 1 & 0 & 1 & 0 & \cdots & 0 \\
\vdots &  \ddots & \ddots & \ddots & \ddots &  & \vdots  \\
\vdots &  &  \ddots & \ddots & \ddots & \ddots & 0 \\
0 & \cdots & \cdots & 0 & 1 & 0 & 1 \\
0 & \cdots & \cdots & \cdots & 0 & 1 & 0
\end{vmatrix}.
\]
We first apply the cofactor expansion along the first row,
and we then apply it again along the first column.
We have $\det(H_{\eta}(P_{n})) = - \det(H_{\eta}(P_{n-2}))$.
Since $\det H_{\eta}(P_1) = 0$ and $\det H_{\eta}(P_2) = -1$,
we have the statement.
\end{proof}

\begin{pro} \label{35}
{\it
We have
\[ \det H_{\eta}(C_n^j) = (-1)^{n+1} 2 \cos (\eta j) + (-1)^{\lfloor \frac{n}{2} \rfloor} (1 + (-1)^n). \]
}
\end{pro}

\begin{proof}
We calculate $\det H_{\eta}(C_{n}^{j})$,
which is
\[
\begin{vmatrix}
0 & e^{\eta i} & 0 & \cdots & \cdots & \cdots  & \cdots & 0 & 1 \\
e^{-\eta i} & 0 & e^{\eta i} & 0 & \cdots & \cdots & \cdots & \cdots & 0 \\
0 &  \ddots & \ddots & \ddots &  & & & &  \vdots  \\
\vdots &  & e^{-\eta i} & 0 & e^{\eta i} & & & & \vdots  \\
\vdots &  &  & e^{-\eta i} & 0 & 1 &  & & \vdots  \\
\vdots &  &  &  & 1 & 0 & 1 & & \vdots  \\
\vdots &  & &  &  & \ddots & \ddots & \ddots & 0 \\
0 & \cdots & \cdots & \cdots & \cdots & 0 & 1 & 0 & 1 \\
1 & 0  & \cdots & \cdots & \cdots & \cdots & 0 & 1 & 0 
\end{vmatrix}.
\]
By the cofactor expansion along the first row,
we have $\det H_{\eta}(C_n^j) = -e^{\eta i}D_1 + (-1)^{n+1}D_2$, where
\[ D_1 =
\begin{vmatrix}
e^{-\eta i} & e^{\eta i} & 0 & \cdots & \cdots & \cdots  & \cdots & \cdots & 0 & 0 \\
0 & 0 & e^{\eta i} & 0 & \cdots & \cdots & \cdots & \cdots & \cdots & 0 \\
\vdots &  e^{-\eta i} & 0 & e^{\eta i} &  & & & & & \vdots  \\
\vdots & & \ddots & \ddots & \ddots &  & & & & \vdots  \\
\vdots & & & e^{-\eta i} & 0 & e^{\eta i} & & & & \vdots  \\
\vdots & & &  & e^{-\eta i} & 0 & 1 &  & & \vdots  \\
\vdots & & &  &  & 1 & 0 & 1 & & \vdots  \\
\vdots & & & &  &  & \ddots & \ddots & \ddots & 0 \\
0 & \cdots & \cdots & \cdots & \cdots & \cdots & 0 & 1 & 0 & 1 \\
1 & 0  & \cdots & \cdots & \cdots & \cdots & \cdots & 0 & 1 & 0 
\end{vmatrix},
\]
and
\[ D_2 =
\begin{vmatrix}
e^{-\eta i} & 0 & e^{\eta i} & 0 & \cdots & \cdots  & \cdots & \cdots & 0 & 0 \\
0  & e^{-\eta i} & 0 & e^{\eta i} & \cdots & \cdots & \cdots & \cdots & \cdots & 0 \\
\vdots & & e^{-\eta i} & 0 & e^{\eta i} &  & & & &  \vdots  \\
\vdots & & & \ddots & \ddots & \ddots &  & & & \vdots  \\
\vdots & & & & e^{-\eta i} & 0 & e^{\eta i} & & & \vdots  \\
\vdots & & & & & e^{-\eta i} & 0 & 1 &  & \vdots  \\
\vdots & & & & &  & 1 & 0 & \ddots & \vdots  \\
\vdots & & & &  &  &  & \ddots & \ddots & 1 \\
0 & \cdots & \cdots & \cdots & \cdots & \cdots &  & 0 & 1 & 0 \\
1 & 0  & \cdots & \cdots & \cdots & \cdots & \cdots &  & 0 & 1 
\end{vmatrix}.
\]
We apply the cofactor expansion along the first column to $D_1$.
By Corollary~\ref{60}, we have
\begin{align*}
D_1 &= (-1)^{1+1} e^{-\eta i} \det H_{\eta}(P_{n-2})+ (-1)^{(n-1)+1} e^{(j-1)\eta i} \\
&= e^{-\eta i} \det H_{\eta}(P_{n-2}) + (-1)^{n} e^{(j-1) \eta i}.
\end{align*}
Applying the cofactor expansion along the first column to $D_2$,
we have
\begin{align*}
D_2 &= (-1)^{1+1}e^{-\eta i} \cdot e^{-(j-1)\eta i} + (-1)^{(n-1) + 1} \det H_{\eta}(P_{n-2}) \\
&= e^{-j \eta i} + (-1)^{n} \det H_{\eta}(P_{n-2}).
\end{align*}
Therefore, by Lemma~\ref{50},
\begin{align*}
\det H_{\eta}(C_n^j) &= -e^{\eta i}D_1 + (-1)^{n+1}D_2 \\
&= (-1)^{n+1} 2 \cos (\eta j) + (-2) \det H_{\eta}(P_{n-2}) \\
&= (-1)^{n+1} 2 \cos (\eta j) + (-2) \cdot (-1)^{\lfloor \frac{n-2}{2} \rfloor} \frac{1+(-1)^{n-2}}{2} \\
&= (-1)^{n+1} 2 \cos (\eta j) + (-1)^{\lfloor \frac{n}{2} \rfloor} (1 + (-1)^n).
\end{align*}
\end{proof}

Proposition~\ref{03} and Proposition~\ref{35} derive Theorem~\ref{Main1},
which is our first main theorem.
In addition,
Proposition~\ref{35} yields that, except for a finite number of $\eta$,
\[ \det (\lambda I - H_{\eta}(C_n^j)) \neq \det (\lambda I - H_{\eta}(C_n^k)) \]
for $j \neq k$.
We supplement the phrase ``except for a finite number of $\eta$" here.
For example, we consider $\eta = \frac{\pi}{2}$ and the mixed cycles of length 4.
We have
\begin{align*}
\det H_{\frac{\pi}{2}}(C_4^j) &= 2 -2\cos \frac{j \pi}{2} \\
&= \begin{cases}
0 \qquad &\text{if $j \in \{0, 4\}$,} \\
2 \qquad &\text{if $j \in \{1, 3\}$,}  \\
4 \qquad &\text{if $j=2$.}
\end{cases}
\end{align*}
As this example points out,
when $\eta \in \MB{Q}\pi$ and the denominator of $\eta / \pi$ is smaller than a given $n$,
the determinants could be equal for different types of mixed cycles.
There are only a finite number of such $\eta$ for given $n$.
This is the reason why the only three types of mixed cycles appeared in the study of \cite{AGNN}.

\subsection{Characteristic polynomials of $H_{\eta}(C_n^j)$}

On the other hand,
if the determinants are same, the characteristic polynomials of mixed cycles are actually equal.
We find the characteristic polynomial of $C_n^j$.
An {\it elementary subgraph} of an undirected graph $\G$ is
a subgraph of $\G$ such that every component is either $K_2$ or an undirected cycle.
Let $\MC{H}(\G)$ denote the set of all the elementary subgraphs of $\G$,
and let $\MC{H}_l(\G)$ denote the set of all the elementary subgraphs of $\G$ on $l$ vertices.
The {\it rank} and {\it corank} of an undirected graph $\G = (V, E)$ are, respectively,
$r(\G) = |V| - c$ and $s(\G) = |E| - |V| + c$,
where $c$ is the number of components of $\G$.
Let $C_n$ be the undirected cycle graph on $n$ vertices.

\begin{pro}
{\it
We have
\[ \det (\lambda I - H_{\eta}(C_n^j)) = \sum_{k = 0}^{\lfloor \frac{n-1}{2} \rfloor}
(-1)^k \frac{n}{n-k} \binom{n-k}{k} \lambda^{n-2k} - 2 \cos (\eta j) + (-1)^{\lfloor \frac{n}{2} \rfloor} (1 + (-1)^n). \]
}
\end{pro}

\begin{proof}
Let
\begin{align*}
\det (\lambda I - H_{\eta}(C_n^j)) &= \lambda^n + a_1 \lambda^{n-1} + \cdots + a_{n-1} \lambda + a_n, \\
\det (\lambda I - H_{\eta}(C_n)) &= \lambda^n + b_1 \lambda^{n-1} + \cdots + b_{n-1} \lambda + b_n.
\end{align*}
Fix $l \in \{1, \dots, n-1 \}$.
Since the girth of $C_n^j$ is $n$,
we have $a_l = b_l$ by Proposition~\ref{10}.
Also,
\begin{equation} \label{35a}
\MC{H}_l(C_n) =
\begin{cases}
\{ \G' \in \MC{H}(C_n) \mid \G' \simeq \tfrac{l}{2} K_2 \} \quad &\text{if $l$ is even,} \\
\emptyset \quad &\text{if $l$ is odd,}
\end{cases}
\end{equation}
where $\G' \simeq \tfrac{l}{2} K_2$ denotes that the graph $\G'$ is isomorphic to
the disjoint union of the $\frac{l}{2}$ complete graphs $K_2$.
Thus, $b_l = 0$ if $l$ is odd.
In addition, if $l$ is even,
by Proposition~7.1 in \cite{B} and (\ref{35a}),
\begin{align*}
b_l &= \sum_{\G' \in \MC{H}_l(C_n)} (-1)^{r(\G')} 2^{s(\G')} \\
&= \sum_{\G' \in \MC{H}_l (C_n)} (-1)^{\frac{l}{2}} \\
&= (-1)^{\frac{l}{2}} | \{ \text{$\tfrac{l}{2}$-matching in $C_n$} \} | \\
&= (-1)^{\frac{l}{2}} \frac{n}{n-\frac{l}{2}} \binom{n-\frac{l}{2}}{\frac{l}{2}}.
\end{align*}
The last equality is given by Exercises in p.14 of \cite{G}.
By Proposition~\ref{35}, the characteristic polynomial $\det (\lambda I - H_{\eta}(C_n^j))$ is
\begin{align*}
& \, \sum_{k = 0}^{\lfloor \frac{n-1}{2} \rfloor} (-1)^k \frac{n}{n-k} \binom{n-k}{k} \lambda^{n-2k}
+ (-1)^n \{ (-1)^{n+1} 2 \cos (\eta j) + (-1)^{\lfloor \frac{n}{2} \rfloor} (1 + (-1)^n) \} \\
=& \sum_{k = 0}^{\lfloor \frac{n-1}{2} \rfloor} (-1)^k \frac{n}{n-k} \binom{n-k}{k} \lambda^{n-2k}
-2 \cos (\eta j) + (-1)^{\lfloor \frac{n}{2} \rfloor + n} (1 + (-1)^n) \\
=& \sum_{k = 0}^{\lfloor \frac{n-1}{2} \rfloor} (-1)^k \frac{n}{n-k} \binom{n-k}{k} \lambda^{n-2k}
-2 \cos (\eta j) + (-1)^{\lfloor \frac{n}{2} \rfloor} (1 + (-1)^n).
\end{align*}
We note that $(-1)^{\lfloor \frac{n}{2} \rfloor + n} \neq (-1)^{\lfloor \frac{n}{2} \rfloor}$,
but $(-1)^{\lfloor \frac{n}{2} \rfloor + n} (1 + (-1)^n) = (-1)^{\lfloor \frac{n}{2} \rfloor} (1 + (-1)^n)$
in the last calculation.
We have the statement.
\end{proof}

\subsection{Mixed graphs as $\MB{T}$-gain graphs}
In this subsection, we briefly touch gain graphs.
The $\eta$-Hermitian adjacency matrix is also seen as an adjacency matrix of a $\MB{T}$-gain graph.
We refer to \cite{MKS, SK} for readers.
We use notations and terminologies in \cite{MKS, SK}.
The $\eta$-Hermitian adjacency matrix of a mixed graph $G = (V, \MC{A})$
is the adjacency matrix of the $\MB{T}$-gain graph $\Phi = (G^{\pm}, \MB{T}, \varphi)$
defined by the $\MB{T}$-gain $\varphi : \MC{A}^{\pm} \to \MB{T}$ such that
\[
\varphi(a) =
\begin{cases}
1 \qquad &\text{if $a \in \MC{A} \cap \MC{A}^{-1}$,} \\
e^{\eta i} \qquad &\text{if $a \in \MC{A} \setminus \MC{A}^{-1}$,} \\
e^{-\eta i} \qquad &\text{if $a \in \MC{A}^{-1} \setminus \MC{A}$,} \\
0 \qquad &\text{otherwise.}
\end{cases}
\]
In \cite{GH, MKS},
the authors mentioned the determinants of the adjacency matrices of $\MB{T}$-gain graphs.
In addition, the coefficients of the characteristic polynomials are also found by using principal minors.
The discussions in this section can also be carried out in terms of $\MB{T}$-gain graphs.

\section{Preliminaries on quantum walks defined by mixed graphs} \label{105}
Let $\eta \in [0, 2\pi)$,
and let $G = (V, \MC{A})$ be a mixed graph.
The {\it $\eta$-function} $\theta : \MC{A}^{\pm} \to \MB{R}$ of a mixed graph $G$ is defined by
\[
\theta(a) = \begin{cases}
\eta \qquad &\text{if $a \in \MC{A} \setminus \MC{A}^{-1}$,} \\
-\eta \qquad &\text{if $a \in \MC{A}^{-1} \setminus \MC{A}$,} \\
0 \qquad &\text{if $a \in \MC{A} \cap \MC{A}^{-1}$.}
\end{cases}
\]
Note that $\theta(a^{-1}) = -\theta(a)$ for any $a \in \MC{A}^{\pm}$.

\subsection{Several matrices on quantum walks defined by mixed graphs}

In \cite{KST},
the authors provided a quantum walk defined by a mixed graph.
Let $G = (V, \MC{A})$ be a mixed graph equipped with an $\eta$-function $\theta$.
We define several matrices (operators) on quantum walks.
The {\it boundary operator} $K = K(G) \in \MB{C}^{V \times \MC{A}^{\pm}}$ is defined by
\[ K_{x,a} = \frac{1}{\sqrt{\deg x}} \delta_{x,t(a)}. \]
The {\it coin operator} $C = C(G) \in \MB{C}^{\MC{A}^{\pm} \times \MC{A}^{\pm}}$ is defined by
$C = 2K^*K-I$.
The {\it shift operator} $S_{\theta} = S_{\theta}(G) \in \MB{C}^{\MC{A}^{\pm} \times \MC{A}^{\pm}}$
is defined by $(S_{\theta})_{a, b} = e^{\theta(b)i}\delta_{a,b^{-1}}$.
Define the {\it time evolution matrix} $U_{\theta} = U_{\theta}(G) \in \MB{C}^{\MC{A}^{\pm} \times \MC{A}^{\pm}}$ by $U_{\theta} = S_{\theta} C$.

\begin{lem} \label{22}
{\it
Let $G = (V, \MC{A})$ be a mixed graph equipped with an $\eta$-function $\theta$.
We have
\[ (U_{\theta})_{a,b} = e^{-\theta(a) i} \marukakko{
  \frac{2}{\deg_{G} t(b)} \delta_{o(a), t(b)} - \delta_{a, b^{-1}}
  } \]
for any $a,b \in \MC{A}^{\pm}$.
}
\end{lem}

\begin{proof}
Indeed,
\begin{align*}
(U_{\theta})_{a,b} &= (2 S_{\theta} K^* K - S_{\theta})_{a,b} \\
&= 2(S_{\theta} K^* K)_{a,b} - (S_{\theta})_{a,b} \\
&= 2 \sum_{z \in \MC{A}} \sum_{x \in V} (S_{\theta})_{a,z} (K^*)_{z, x} K_{x,b} - e^{\theta(b)i} \delta_{a, b^{-1}} \\
&= 2 \sum_{z \in \MC{A}} \sum_{x \in V} e^{\theta(z)i} \frac{1}{\sqrt{\deg x}} \frac{1}{\sqrt{\deg x}} \delta_{a,z^{-1}} \delta_{x, t(z)} \delta_{x, t(b)} - e^{\theta(a^{-1})i} \delta_{a, b^{-1}} \\
&= 2 \sum_{x \in V} e^{\theta(a^{-1})i} \frac{1}{\deg x} \delta_{x, t(a^{-1})} \delta_{x, t(b)} - e^{\theta(a^{-1})i}  \delta_{a, b^{-1}} \\
&= \frac{2  e^{\theta(a^{-1})i}}{\deg_{G} t(b)} \delta_{o(a), t(b)} - e^{\theta(a^{-1})i}  \delta_{a, b^{-1}} \\
&= e^{- \theta(a) i} \marukakko{ \frac{2}{\deg_{G} t(b)} \delta_{o(a), t(b)} - \delta_{a, b^{-1}} }.
\end{align*}
\end{proof}


The following is an important theorem that links quantum walks and spectral graph theory.
We cite \cite{KST}.
In \cite{HKSS2014}, Higuchi et al~proved a similar claim in more general models.

\begin{thm}[\cite{KST}] \label{21}
{\it
Let $G = (V, \MC{A})$ be a mixed graph equipped with an $\eta$-function $\theta$,
and let $U_{\theta}$ be the time evolution matrix.
Then we have
\[ \Spec(U_{\theta}) = \{ e^{\pm i \cos^{-1} (\lambda)} \mid \lambda \in \Spec(\NH_{\eta}(G)) \} \cup \{ 1 \}^{M_1} \cup \{-1\}^{M_{-1}}, \]
where
\begin{align*}
M_{1} = \frac{1}{2}|\MC{A}^{\pm}| - |V| + \dim \ker ( \NH_{\eta}(G) - I), \\
M_{-1} = \frac{1}{2}|\MC{A}^{\pm}| - |V| + \dim \ker ( \NH_{\eta}(G) + I).
\end{align*}
}
\end{thm}

The operators (matrices) used in our quantum walks are summarized in Table~\ref{1000},
where $G = (V, \MC{A})$ is a mixed graph equipped with an $\eta$-function $\theta$.

\begin{table}[H]
  \centering
  \begin{tabular}{|c|c|c|c|}
\hline
Notation & Name & Indices of rows and columns & Definition \\
\hline
\hline
$K$ & Boundary & $V \times \MC{A}^{\pm}$ & $K_{x,a} =  \frac{1}{ \sqrt{\deg x} } \delta_{x, t(a)}$ \\
\hline
$C$ & Coin & $\MC{A}^{\pm} \times \MC{A}^{\pm}$ &$ C = 2K^*K - I$ \\
\hline
$S_{\theta}$ & Shift & $\MC{A}^{\pm} \times \MC{A}^{\pm}$ & $(S_{\theta})_{ab} = e^{\theta(b)i}\delta_{a,b^{-1}}$ \\
\hline
$U_{\theta}$ & Time evolution & $\MC{A}^{\pm} \times \MC{A}^{\pm}$  & $U_{\theta} = S_{\theta} C$ \\
\hline
\end{tabular}
 \caption{The operators (matrices) used in our quantum walk} \label{1000}
\end{table}

\subsection{Necessary and sufficient conditions on periodicity}
Let $U_{\theta}$ be a time evolution matrix of a mixed graph $G$ equipped with an $\eta$-function $\theta$.
We say that $G$ is {\it periodic} if there exists $\tau \in \MB{N}$ such that $U_{\theta}^{\tau} = I$.
When the mixed graph $G$ is periodic,
the {\it period} is defined by $\min \{ \tau \in \MB{N} \mid U_{\theta}^{\tau} = I \}$.

\begin{lem} \label{13}
{\it
Let $U_{\theta}$ be a time evolution matrix of a mixed graph $G$ equipped with an $\eta$-function $\theta$.
Then, we have
\begin{equation} \label{70}
\{ \tau \in \MB{N} \mid U_{\theta}^{\tau} = I \} =
\{ \tau \in \MB{N} \mid \lambda^{\tau} = 1 \text{ \emph{for any} $\lambda \in \Spec(U_{\theta})$} \}.
\end{equation}
In particular, $G$ is periodic if and only if
there exists $\tau \in \MB{N}$ such that $\lambda^{\tau} = 1$ holds
for any eigenvalue $\lambda$ of $U_{\theta}$.
}
\end{lem}

\begin{proof}
Since $U_{\theta}$ is unitary,
there exists a unitary matrix $Q$ such that
\[ Q^{*} U_{\theta} Q = \diag (\lambda_1, \cdots, \lambda_{2m}), \]
where $m$ is the number of edges of $G^{\pm}$.
Thus for $\tau \in \MB{N}$,
\[ Q^{*} U_{\theta}^{\tau} Q = \diag (\lambda_1^{\tau}, \cdots, \lambda_{2m}^{\tau}). \]
This implies~(\ref{70}).
\end{proof}

Define $\BM{e}^{(a)} \in \MB{C}^{\MC{A}^{\pm}}$ by $(\BM{e}^{(a)})_z = \delta_{a,z}$.
Let $\MC{E}_{\MC{A}^{\pm}} = \{ \BM{e}^{(a)} \mid a \in \MC{A}^{\pm} \}$.
This is the canonical basis of $\MB{C}^{\MC{A}^{\pm}}$.

\begin{lem} \label{55}
{\it
Let $U_{\theta}$ be a time evolution matrix of a mixed graph $G$ equipped with an $\eta$-function $\theta$.
Then, we have
\[ \{ \tau \in \MB{N} \mid U_{\theta}^{\tau} = I \} =
\{ \tau \in \MB{N} \mid U_{\theta}^{\tau}\BM{e}^{(a)} = \BM{e}^{(a)}
\text{ \emph{for any} $\BM{e}^{(a)} \in \MC{E}_{\MC{A}^{\pm}}$} \}.\]
In particular,
$G$ is periodic if and only if
there exists $\tau \in \MB{N}$ such that $U_{\theta}^{\tau}\BM{e}^{(a)} = \BM{e}^{(a)}$ holds
for any vector $\BM{e}^{(a)} \in \MC{E}_{\MC{A}^{\pm}}$.
}
\end{lem}

\begin{proof}
It is clear that the left-hand side is included in the right-hand side.
We show the reverse inclusion.
Suppose $U_{\theta}^{\tau}\BM{e}^{(a)} = \BM{e}^{(a)}$ for any vector $\BM{e}^{(a)} \in \MC{E}_{\MC{A}^{\pm}}$ and for some $\tau \in \MB{N}$.
Number the arc set $\MC{A^{\pm}}$ as $a_1, a_2, \dots, a_{2m}$,
where $m$ is the number of edges of $G^{\pm}$.
Let $P = [ \BM{e}^{(a_1)} \, \BM{e}^{(a_2)} \, \cdots \BM{e}^{(a_{2m})} ]$.
Then, we have $U_{\theta}^{\tau} P = P$.
Since $P$ is invertible, $U_{\theta}^{\tau}  = I$ holds.
\end{proof}

\section{Periodicity of mixed paths} \label{106}

In this section,
we discuss periodicity of mixed paths.
As a preparation,
we introduce notations for expressing dynamics of quantum walk.


Let $G = (V, \MC{A})$ be a mixed graph equipped with an $\eta$-function $\theta$.
We write the components of a vector $\Psi \in \MB{C}^{\MC{A}^{\pm}}$
on the arcs of the graph as in Figure~\ref{48}.
If a component of $\Psi$ is $0$,
we omit the arc itself corresponding to the component.
If a component of $\Psi$ is $1$,
we may omit the value on the arc corresponding to the component.

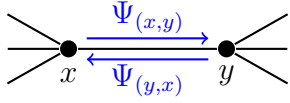
\begin{figure}[ht]
\begin{center}
\begin{tikzpicture}
[scale = 0.7,
line width = 0.8pt,
v/.style = {circle, fill = black, inner sep = 0.8mm},u/.style = {circle, fill = white, inner sep = 0.1mm}]
  \node[u] (1) at (-1.2, 0) {};
  \node[v] (2) at (0, 0) {};
  \node[v] (3) at (3, 0) {};
  \node[u] (4) at (-1.2, 0.68) {};
  \node[u] (5) at (-1.2, -0.68) {};
  \node[u] (7) at (4.2, 0) {};
  \node[u] (8) at (4.2, 0.68) {};
  \node[u] (9) at (4.2, -0.68) {};
  \node[u] (12) at (3.3, 0.2) {};
  \node[u] (13) at (5.7, 0.2) {};
  \draw (0,-0.5) node{$x$};
  \draw (3,-0.5) node{$y$};
  \draw (1) to (2);
  \draw[-] (2) to (4);
  \draw[-] (2) to (3);
  \draw[-] (5) to (2);
  \node[u] (10) at (0.3, 0.2) {};
  \node[u] (11) at (2.7, 0.2) {};
  \draw[draw= blue,->] (10) to (11);
  \node[u] (20) at (0.3, -0.2) {};
  \node[u] (21) at (2.7, -0.2) {};
  \draw[draw= blue,->] (21) to (20);
  \draw[-] (3) to (7);
  \draw[-] (3) to (8);
  \draw[-] (3) to (9);
  \draw (1.5,0.6) node[blue]{$\Psi_{(x,y)}$};
  \draw (1.5,-0.6) node[blue]{$\Psi_{(y,x)}$};
\end{tikzpicture}
\caption{The components of a vector written on the arcs of the graph} \label{48}
\end{center}
\end{figure}

We provide an example.
Let $G = (V, \MC{A})$ be the mixed graph in Figure~\ref{a01}.
We consider the $\eta$-function $\theta$ for $\eta \in [0, 2\pi)$.
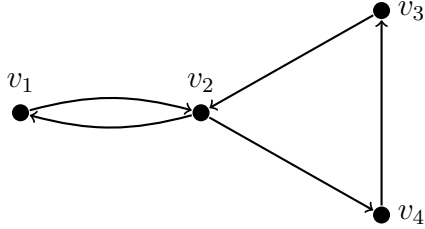
\begin{figure}[ht]
\begin{center}
\begin{tikzpicture}
[scale = 0.8,
line width = 0.8pt,
v/.style = {circle, fill = black, inner sep = 0.8mm}]
  \node[v] (1) at (-3, 0) {};
  \draw (-3,0.5) node {$v_1$};
  \node[v] (2) at (0, 0) {};
  \draw (0,0.5) node {$v_2$};
  \node[v] (3) at (3, 1.7) {};
  \draw (3.5,1.7) node {$v_3$};
  \node[v] (4) at (3, -1.7) {};
  \draw (3.5,-1.7) node {$v_4$};
  \draw[->] (1) to [bend left = 15] (2);
  \draw[->] (2) to [bend left = 15] (1);
  \draw[->] (2) to (4);
  \draw[->] (4) to (3);
  \draw[->] (3) to (2);
\end{tikzpicture}
\caption{Mixed graph $G$} \label{a01}
\end{center}
\end{figure}
We focus on the vector $\BM{e}^{((v_3, v_2))} \in \MC{E}_{\MC{A}^{\pm}}$.
The actions of the coin operator $C$ and the shift operator $S_{\theta}$ are shown in Figure~\ref{40} and~\ref{41}.
Since $U_{\theta} = S_{\theta} C$,
the action of the time evolution matrix $U_{\theta}$ is as in Figure~\ref{42}.

\begin{figure}[ht]
\begin{center}
\begin{tikzpicture}
[scale = 0.7,
line width = 0.8pt,
v/.style = {circle, fill = black, inner sep = 0.8mm},u/.style = {circle, fill = white, inner sep = 0.1mm}]
  \node[v] (1) at (-3, 0) {};
  \draw (-3,0.5) node {$v_1$};
  \node[v] (2) at (0, 0) {};
  \draw (0,0.5) node {$v_2$};
  \node[v] (3) at (3, 1.7) {};
  \draw (3.5,1.7) node {$v_3$};
  \node[v] (4) at (3, -1.7) {};
  \draw (3.5,-1.7) node {$v_4$};
  \node[u] (5) at (0.3, 0.37) {};
  \node[u] (6) at (2.7, 1.73) {};
  \draw (1) to (2);
  \draw[-] (2) to (4);
  \draw[-] (4) to (3);
  \draw[-] (3) to (2);
  \draw[draw= blue,->] (6) to (5);
\end{tikzpicture}
\raisebox{13mm}{$\quad \overset{C}{\mapsto} \quad$}
\begin{tikzpicture}
[scale = 0.7,
line width = 0.8pt,
v/.style = {circle, fill = black, inner sep = 0.8mm},u/.style = {circle, fill = white, inner sep = 0.1mm}]
  \node[v] (1) at (-3, 0) {};
  \draw (-3,0.5) node {$v_1$};
  \node[v] (2) at (0, 0) {};
  \draw (0,0.5) node {$v_2$};
  \node[v] (3) at (3, 1.7) {};
  \draw (3.5,1.7) node {$v_3$};
  \node[v] (4) at (3, -1.7) {};
  \draw (3.5,-1.7) node {$v_4$};
  \node[u] (5) at (0.3, 0.37) {};
  \node[u] (6) at (2.7, 1.73) {};
  \node[u] (7) at (-2.7, 0.2) {};
  \node[u] (8) at (-0.3, 0.2) {};
  \node[u] (9) at (0.3, -0.37) {};
  \node[u] (10) at (2.7, -1.73) {};
  \draw (1) to (2);
  \draw[-] (2) to (4);
  \draw[-] (4) to (3);
  \draw[-] (3) to (2);
  \draw[draw= blue,->] (6) to (5);
  \draw[draw= blue,->] (7) to (8);
  \draw[draw= blue,->] (10) to (9);
  \draw (-1.5,0.75) node[blue] {$\tfrac{2}{3}$};
  \draw (1.1,1.5) node[blue] {$-\tfrac{1}{3}$};
  \draw (1.2,-1.5) node[blue] {$\tfrac{2}{3}$};
\end{tikzpicture}
\caption{Action of $C$} \label{40}
\end{center}
\end{figure}
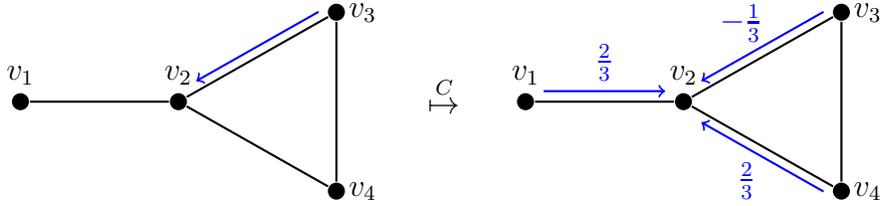

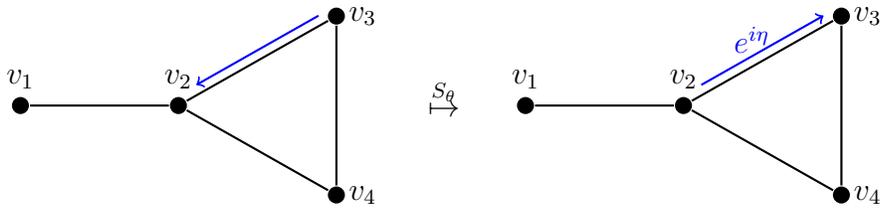
\begin{figure}[ht]
\begin{center}
\begin{tikzpicture}
[scale = 0.7,
line width = 0.8pt,
v/.style = {circle, fill = black, inner sep = 0.8mm},u/.style = {circle, fill = white, inner sep = 0.1mm}]
  \node[v] (1) at (-3, 0) {};
  \draw (-3,0.5) node {$v_1$};
  \node[v] (2) at (0, 0) {};
  \draw (0,0.5) node {$v_2$};
  \node[v] (3) at (3, 1.7) {};
  \draw (3.5,1.7) node {$v_3$};
  \node[v] (4) at (3, -1.7) {};
  \draw (3.5,-1.7) node {$v_4$};
  \node[u] (5) at (0.3, 0.37) {};
  \node[u] (6) at (2.7, 1.73) {};
  \draw (1) to (2);
  \draw[-] (2) to (4);
  \draw[-] (4) to (3);
  \draw[-] (3) to (2);
  \draw[draw= blue,->] (6) to (5);
\end{tikzpicture}
\raisebox{13mm}{$\quad \overset{S_{\theta}}{\mapsto} \quad$}
\begin{tikzpicture}
[scale = 0.7,
line width = 0.8pt,
v/.style = {circle, fill = black, inner sep = 0.8mm},u/.style = {circle, fill = white, inner sep = 0.1mm}]
  \node[v] (1) at (-3, 0) {};
  \draw (-3,0.5) node {$v_1$};
  \node[v] (2) at (0, 0) {};
  \draw (0,0.5) node {$v_2$};
  \node[v] (3) at (3, 1.7) {};
  \draw (3.5,1.7) node {$v_3$};
  \node[v] (4) at (3, -1.7) {};
  \draw (3.5,-1.7) node {$v_4$};
  \node[u] (5) at (0.3, 0.37) {};
  \node[u] (6) at (2.7, 1.73) {};
  \draw (1.3,1.25) node[blue]{$e^{i\eta }$};
  \draw (1) to (2);
  \draw[-] (2) to (4);
  \draw[-] (4) to (3);
  \draw[-] (3) to (2);
  \draw[draw= blue,->] (5) to (6);
\end{tikzpicture}
\caption{Action of $S_{\theta}$} \label{41}
\end{center}
\end{figure}

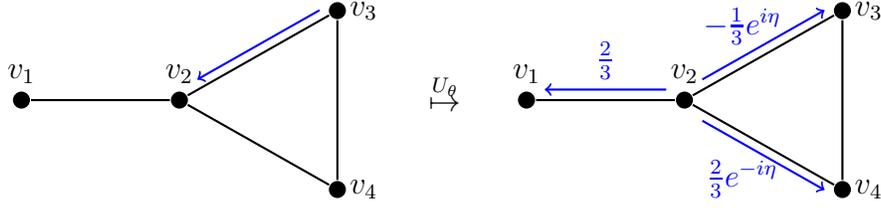
\begin{figure}[ht]
\begin{center}
\begin{tikzpicture}
[scale = 0.7,
line width = 0.8pt,
v/.style = {circle, fill = black, inner sep = 0.8mm},u/.style = {circle, fill = white, inner sep = 0.1mm}]
  \node[v] (1) at (-3, 0) {};
  \draw (-3,0.5) node {$v_1$};
  \node[v] (2) at (0, 0) {};
  \draw (0,0.5) node {$v_2$};
  \node[v] (3) at (3, 1.7) {};
  \draw (3.5,1.7) node {$v_3$};
  \node[v] (4) at (3, -1.7) {};
  \draw (3.5,-1.7) node {$v_4$};
  \node[u] (5) at (0.3, 0.37) {};
  \node[u] (6) at (2.7, 1.73) {};
  \draw (1) to (2);
  \draw[-] (2) to (4);
  \draw[-] (4) to (3);
  \draw[-] (3) to (2);
  \draw[draw= blue,->] (6) to (5);
\end{tikzpicture}
\raisebox{13mm}{$\quad \overset{U_{\theta}}{\mapsto} \quad$}
\begin{tikzpicture}
[scale = 0.7,
line width = 0.8pt,
v/.style = {circle, fill = black, inner sep = 0.8mm},u/.style = {circle, fill = white, inner sep = 0.1mm}]
  \node[v] (1) at (-3, 0) {};
  \draw (-3,0.5) node {$v_1$};
  \node[v] (2) at (0, 0) {};
  \draw (0,0.5) node {$v_2$};
  \node[v] (3) at (3, 1.7) {};
  \draw (3.5,1.7) node {$v_3$};
  \node[v] (4) at (3, -1.7) {};
  \draw (3.5,-1.7) node {$v_4$};
  \node[u] (5) at (0.3, 0.37) {};
  \node[u] (6) at (2.7, 1.73) {};
  \node[u] (7) at (-2.7, 0.2) {};
  \node[u] (8) at (-0.3, 0.2) {};
  \node[u] (9) at (0.3, -0.37) {};
  \node[u] (10) at (2.7, -1.73) {};
  \draw (1.1,1.4) node[blue]{$-\tfrac{1}{3}e^{i\eta}$};
  \draw (-1.5,0.75) node[blue]{$\tfrac{2}{3}$};
  \draw (1.1,-1.5) node[blue]{$\tfrac{2}{3} e^{-i\eta}$};
  \draw (1) to (2);
  \draw[-] (2) to (4);
  \draw[-] (4) to (3);
  \draw[-] (3) to (2);
  \draw[draw= blue,->] (5) to (6);
  \draw[draw= blue,->] (8) to (7);
  \draw[draw= blue,->] (9) to (10);
\end{tikzpicture}
\caption{Action of $U_{\theta}$} \label{42}
\end{center}
\end{figure}

\begin{lem} \label{25}
{\it
Let $G = (V, \MC{A})$ be a mixed graph equipped with an $\eta$-function $\theta$,
and let $a \in \MC{A}^{\pm}$.
We have the following:
\begin{enumerate}[(1)]
\item If $\deg t(a) = 1$, then $U_{\theta} \BM{e}^{(a)} = e^{\theta(a) i}\BM{e}^{(a^{-1})}$; and
\item If $\deg t(a) = 2$, then $U_{\theta} \BM{e}^{(a)} = e^{-\theta(b)i}\BM{e}^{(b)}$,
where $b$ is the arc in $\{ z \in \MC{A}^{\pm} \mid o(z) = t(a) \} \setminus \{a^{-1}\}$.
\end{enumerate}
}
\end{lem}

\begin{proof}
Let $U_{\theta} = U_{\theta}(G)$.
First, by Lemma~\ref{22}, we have
\begin{align*}
(U_{\theta} \BM{e}^{(a)})_{z}
&= \sum_{w \in \MC{A}^{\pm}} (U_{\theta})_{z,w} (\BM{e}^{(a)})_{w} \\
&= (U_{\theta})_{z,a} \\
&= e^{-\theta(z) i} \marukakko{ \frac{2}{\deg_{G} t(a)} \delta_{o(z), t(a)} - \delta_{z, a^{-1}} }.
\end{align*}
Consider the case of $\deg t(a) = 1$.
Then $o(z) = t(a)$ if and only if $z = a^{-1}$.
Thus,
\[ (U_{\theta} \BM{e}^{(a)})_{z} = e^{-\theta(a^{-1}) i} ( 2 \delta_{z, a^{-1}} - \delta_{z, a^{-1}} )
= e^{\theta(a) i} \delta_{z, a^{-1}}.
\]
We have $U_{\theta} \BM{e}^{(a)} = e^{\theta(a) i}\BM{e}^{(a^{-1})}$.
We next consider the case of $\deg t(a) = 2$.
Then,
\begin{align*}
(U_{\theta} e_a)_{z} &= e^{-\theta(z) i} (\delta_{o(z), t(a)} - \delta_{z, a^{-1}}) \\
&= \begin{cases}
0 \quad &\text{if $z=a^{-1}$,} \\
e^{-\theta(b) i} \quad &\text{if $z = b$,} \\
0 \quad &\text{otherwise.}
\end{cases}
\end{align*}
Therefore, we have $U_{\theta} \BM{e}^{(a)} = e^{-\theta(b)i}\BM{e}^{(b)}$.
\end{proof}
The above lemma is illustrated as in Figure~\ref{43} and~\ref{44}.
Now, we discuss periodicity of mixed paths.
The following is our second main theorem.
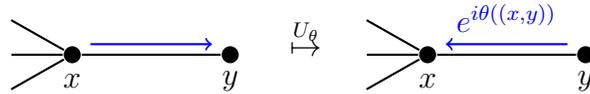
\begin{figure}[ht]
\begin{center}
\begin{tikzpicture}
[scale = 0.7,
line width = 0.8pt,
v/.style = {circle, fill = black, inner sep = 0.8mm},u/.style = {circle, fill = white, inner sep = 0.1mm}]
  \node[u] (1) at (-1.2, 0) {};
  \node[v] (2) at (0, 0) {};
  \node[v] (3) at (3, 0) {};
  \node[u] (4) at (-1.2, 0.68) {};
  \node[u] (5) at (-1.2, -0.68) {};
  \node[u] (6) at (0.3, 0.2) {};
  \node[u] (7) at (2.7, 0.2) {};
  \draw (0,-0.5) node{$x$};
  \draw (3,-0.5) node{$y$};
  \draw (1) to (2);
  \draw[-] (2) to (4);
  \draw[-] (2) to (3);
  \draw[-] (5) to (2);
  \draw[draw= blue,->] (6) to (7);
\end{tikzpicture}
\raisebox{6mm}{$\quad \overset{U_{\theta}}{\mapsto} \quad$}
\begin{tikzpicture}
[scale = 0.7,
line width = 0.8pt,
v/.style = {circle, fill = black, inner sep = 0.8mm},u/.style = {circle, fill = white, inner sep = 0.1mm}]
  \node[u] (1) at (-1.2, 0) {};
  \node[v] (2) at (0, 0) {};
  \node[v] (3) at (3, 0) {};
  \node[u] (4) at (-1.2, 0.68) {};
  \node[u] (5) at (-1.2, -0.68) {};
  \node[u] (6) at (0.3, 0.2) {};
  \node[u] (7) at (2.7, 0.2) {};
  \draw (1.5,0.7) node[blue]{$e^{i\theta((x,y)) }$};
  \draw (0,-0.5) node{$x$};
  \draw (3,-0.5) node{$y$};
  \draw (1) to (2);
  \draw[-] (2) to (4);
  \draw[-] (2) to (3);
  \draw[-] (5) to (2);
  \draw[draw= blue,->] (7) to (6);
\end{tikzpicture}
\caption{Illustration of Lemma~\ref{25} (1)} \label{43}
\end{center}
\end{figure}

\begin{figure}[ht]
\begin{center}
\begin{tikzpicture}
[scale = 0.7,
line width = 0.8pt,
v/.style = {circle, fill = black, inner sep = 0.8mm},u/.style = {circle, fill = white, inner sep = 0.1mm}]
  \node[u] (1) at (-1.2, 0) {};
  \node[v] (2) at (0, 0) {};
  \node[v] (3) at (3, 0) {};
  \node[u] (4) at (-1.2, 0.68) {};
  \node[u] (5) at (-1.2, -0.68) {};
  \node[v] (6) at (6, 0) {};
  \node[u] (7) at (7.2, 0) {};
  \node[u] (8) at (7.2, 0.68) {};
  \node[u] (9) at (7.2, -0.68) {};
  \node[u] (10) at (0.3, 0.2) {};
  \node[u] (11) at (2.7, 0.2) {};
  \node[u] (12) at (3.3, 0.2) {};
  \node[u] (13) at (5.7, 0.2) {};
  \draw (0,-0.5) node{$x$};
  \draw (3,-0.5) node{$y$};
  \draw (6,-0.5) node{$z$};
  \draw (1) to (2);
  \draw[-] (2) to (4);
  \draw[-] (2) to (3);
  \draw[-] (5) to (2);
  \draw[draw= blue,->] (10) to (11);
  \draw[-] (3) to (6);
  \draw[-] (6) to (7);
  \draw[-] (6) to (8);
  \draw[-] (6) to (9);
\end{tikzpicture}
\raisebox{6mm}{$\quad \overset{U_{\theta}}{\mapsto} \quad$}
\begin{tikzpicture}
[scale = 0.7,
line width = 0.8pt,
v/.style = {circle, fill = black, inner sep = 0.8mm},u/.style = {circle, fill = white, inner sep = 0.1mm}]
  \node[u] (1) at (-1.2, 0) {};
  \node[v] (2) at (0, 0) {};
  \node[v] (3) at (3, 0) {};
  \node[u] (4) at (-1.2, 0.68) {};
  \node[u] (5) at (-1.2, -0.68) {};
  \node[v] (6) at (6, 0) {};
  \node[u] (7) at (7.2, 0) {};
  \node[u] (8) at (7.2, 0.68) {};
  \node[u] (9) at (7.2, -0.68) {};
  \node[u] (10) at (0.3, 0.2) {};
  \node[u] (11) at (2.7, 0.2) {};
  \node[u] (12) at (3.3, 0.2) {};
  \node[u] (13) at (5.7, 0.2) {};
  \draw (4.5,0.7) node[blue]{$e^{-i\theta((y,z)) }$};
  \draw (0,-0.5) node{$x$};
  \draw (3,-0.5) node{$y$};
  \draw (6,-0.5) node{$z$};
  \draw (1) to (2);
  \draw[-] (2) to (4);
  \draw[-] (2) to (3);
  \draw[-] (5) to (2);
  \draw[draw= blue,->] (12) to (13);
  \draw[-] (3) to (6);
  \draw[-] (6) to (7);
  \draw[-] (6) to (8);
  \draw[-] (6) to (9);
\end{tikzpicture}
\caption{Illustration of Lemma~\ref{25} (2)} \label{44}
\end{center}
\end{figure}

\begin{thm}
{\it
Let $G = (V, \MC{A})$ be a mixed path on $n$ vertices equipped with an $\eta$-function $\theta$.
Then $G$ is periodic for any $\eta \in [0, 2\pi)$, and the period is $2(n-1)$.
}
\end{thm}

\begin{proof}
By Proposition~\ref{12},
$G$ and $G^{\pm}$ are $\NH_{\eta}$-cospectral since $G$ is a mixed tree.
By Theorem~\ref{21},
$U_{\theta}(G)$ and $U_{\theta}(G^{\pm})$ have the same eigenvalues.
From Lemma~\ref{13},
periodicity of $G$ and its period are determined by the eigenvalues of the time evolution matrix.
Thus, it is sufficient to discuss only periodicity of $G^{\pm}$,
which is the undirected path graph $P_n$ on $n$ vertices.
By Lemma~\ref{25}, we have
$U_{\theta}(P_n)^{2(n-1)} \BM{e}^{(a)} = \BM{e}^{(a)}$
for any vector $\BM{e}^{(a)} \in \MC{E}_{\MC{A}^{\pm}}$.
The dynamics is as follows:
\begin{align*}
\begin{tikzpicture}
[scale = 0.8,
line width = 0.8pt,
v/.style = {circle, fill = black, inner sep = 0.8mm},u/.style = {circle, fill = white, inner sep = 0.1mm}]
  \node[v] (1) at (-2, 0) {};
  \node[v] (2) at (0, 0) {};
  \node[v] (3) at (2, 0) {};
  \node[u] (4) at (2.5, 0) {};
  \draw (2.75,0) node{$\dots$};
  \node[u] (5) at (3, 0) {};
  \node[v] (6) at (3.5, 0) {};
  \node[v] (7) at (5.5, 0) {};
  \node[u] (8) at (-1.8, 0.2) {};
  \node[u] (9) at (-0.2, 0.2) {};
  \draw (1) to (2);
  \draw[-] (2) to (3);
  \draw[-] (3) to (4);
  \draw[-] (5) to (6);
  \draw[-] (6) to (7);
  \draw[draw= blue,->] (8) to (9);
\end{tikzpicture}
& \raisebox{1mm}{$\quad \overset{U_{\theta}}{\mapsto} \quad$}
\begin{tikzpicture}
[scale = 0.8,
line width = 0.8pt,
v/.style = {circle, fill = black, inner sep = 0.8mm},u/.style = {circle, fill = white, inner sep = 0.1mm}]
  \node[v] (1) at (-2, 0) {};
  \node[v] (2) at (0, 0) {};
  \node[v] (3) at (2, 0) {};
  \node[u] (4) at (2.5, 0) {};
  \draw (2.75,0) node{$\dots$};
  \node[u] (5) at (3, 0) {};
  \node[v] (6) at (3.5, 0) {};
  \node[v] (7) at (5.5, 0) {};
  \node[u] (8) at (0.2, 0.2) {};
  \node[u] (9) at (1.8, 0.2) {};
  \draw (1) to (2);
  \draw[-] (2) to (3);
  \draw[-] (3) to (4);
  \draw[-] (5) to (6);
  \draw[-] (6) to (7);
  \draw[draw= blue,->] (8) to (9);
\end{tikzpicture} \\
& \raisebox{1mm}{$\quad \overset{U_{\theta}}{\mapsto} \quad \cdots$} \\
& \quad \hspace{2mm} \vdots \\
& \raisebox{1mm}{$\quad \overset{U_{\theta}}{\mapsto} \quad$}
\begin{tikzpicture}
[scale = 0.8,
line width = 0.8pt,
v/.style = {circle, fill = black, inner sep = 0.8mm},u/.style = {circle, fill = white, inner sep = 0.1mm}]
  \node[v] (1) at (-2, 0) {};
  \node[v] (2) at (0, 0) {};
  \node[v] (3) at (2, 0) {};
  \node[u] (4) at (2.5, 0) {};
  \draw (2.75,0) node{$\dots$};
  \node[u] (5) at (3, 0) {};
  \node[v] (6) at (3.5, 0) {};
  \node[v] (7) at (5.5, 0) {};
  \node[u] (8) at (3.7, 0.2) {};
  \node[u] (9) at (5.3, 0.2) {};
  \draw (1) to (2);
  \draw[-] (2) to (3);
  \draw[-] (3) to (4);
  \draw[-] (5) to (6);
  \draw[-] (6) to (7);
  \draw[draw= blue,->] (8) to (9);
\end{tikzpicture} \\
&\raisebox{1mm}{$\quad \overset{U_{\theta}}{\mapsto} \quad$}
\begin{tikzpicture}
[scale = 0.8,
line width = 0.8pt,
v/.style = {circle, fill = black, inner sep = 0.8mm},u/.style = {circle, fill = white, inner sep = 0.1mm}]
  \node[v] (1) at (-2, 0) {};
  \node[v] (2) at (0, 0) {};
  \node[v] (3) at (2, 0) {};
  \node[u] (4) at (2.5, 0) {};
  \draw (2.75,0) node{$\dots$};
  \node[u] (5) at (3, 0) {};
  \node[v] (6) at (3.5, 0) {};
  \node[v] (7) at (5.5, 0) {};
  \node[u] (8) at (5.3, -0.2) {};
  \node[u] (9) at (3.7, -0.2) {};
  \draw (1) to (2);
  \draw[-] (2) to (3);
  \draw[-] (3) to (4);
  \draw[-] (5) to (6);
  \draw[-] (6) to (7);
  \draw[draw= blue,->] (8) to (9);
\end{tikzpicture} \\
& \raisebox{1mm}{$\quad \overset{U_{\theta}}{\mapsto} \quad \cdots$} \\
& \quad \hspace{2mm} \vdots \\
& \raisebox{1mm}{$\quad \overset{U_{\theta}}{\mapsto} \quad$}
\begin{tikzpicture}
[scale = 0.8,
line width = 0.8pt,
v/.style = {circle, fill = black, inner sep = 0.8mm},u/.style = {circle, fill = white, inner sep = 0.1mm}]
  \node[v] (1) at (-2, 0) {};
  \node[v] (2) at (0, 0) {};
  \node[v] (3) at (2, 0) {};
  \node[u] (4) at (2.5, 0) {};
  \draw (2.75,0) node{$\dots$};
  \node[u] (5) at (3, 0) {};
  \node[v] (6) at (3.5, 0) {};
  \node[v] (7) at (5.5, 0) {};
  \node[u] (8) at (-1.8, -0.2) {};
  \node[u] (9) at (-0.2, -0.2) {};
  \draw (1) to (2);
  \draw[-] (2) to (3);
  \draw[-] (3) to (4);
  \draw[-] (5) to (6);
  \draw[-] (6) to (7);
  \draw[draw= blue,->] (9) to (8);
\end{tikzpicture} \\
& \raisebox{1mm}{$\quad \overset{U_{\theta}}{\mapsto} \quad$}
\begin{tikzpicture}
[scale = 0.8,
line width = 0.8pt,
v/.style = {circle, fill = black, inner sep = 0.8mm},u/.style = {circle, fill = white, inner sep = 0.1mm}]
  \node[v] (1) at (-2, 0) {};
  \node[v] (2) at (0, 0) {};
  \node[v] (3) at (2, 0) {};
  \node[u] (4) at (2.5, 0) {};
  \draw (2.75,0) node{$\dots$};
  \node[u] (5) at (3, 0) {};
  \node[v] (6) at (3.5, 0) {};
  \node[v] (7) at (5.5, 0) {};
  \node[u] (8) at (-1.8, 0.2) {};
  \node[u] (9) at (-0.2, 0.2) {};
  \draw (1) to (2);
  \draw[-] (2) to (3);
  \draw[-] (3) to (4);
  \draw[-] (5) to (6);
  \draw[-] (6) to (7);
  \draw[draw= blue,->] (8) to (9);
\end{tikzpicture}
\end{align*}
By Lemma~\ref{55},
we see that $G$ is periodic whose period is $2(n-1)$.
\end{proof}

Note that the periodicity of the undirected paths has actually studied
in \cite{KSTY2018} by eigenvalue analysis.
In this paper, we have proven the same fact in a different way.

\section{Periodicity of mixed cycles} \label{107}

Finally, we discuss periodicity of mixed cycles.
Strategy is similar to the mixed paths, but the discussion is more complicated.
Recall that 
a mixed cycle $G$ on $n$ vertices is $H_{\eta}$-cospectral with $C_n^j$ for some $j \in \{0,1,\dots, n\}$ by Proposition~\ref{03}.
For $m_1, m_2 \in \MB{Z}$,
the greatest common divisor of $m_1$ and $m_2$ is denoted by $(m_1, m_2)$.
Note that $(0, m_1) = m_1$.
The following is our third main theorem.

\begin{thm}
{\it
Let $G = (V, \MC{A})$ be a mixed cycle on $n$ vertices equipped with an $\eta$-function $\theta$.
Then,
$G$ is periodic if and only if $\eta \in \MB{Q}\pi$.
In addition,
we suppose that $\eta \in \MB{Q}\pi$ and the mixed cycle $G$ is $H_{\eta}$-cospectral with $C_n^j$.
Let $\eta = \frac{p}{q}\pi$, where $p$ and $q$ are coprime.
Then, the period $\tau$ of $G$ is the following:
\begin{equation} \label{65}
\tau =
\begin{cases}
\frac{2qn}{(j, 2q)} \quad &\text{if $p$ is odd,} \\
\frac{qn}{(j, q)} \quad &\text{if $p$ is even.} \\
\end{cases}
\end{equation}
}
\end{thm}

\begin{proof}
By Proposition~\ref{03},
$G$ is $H_{\eta}$-cospectral with $C_n^j$ for some $j \in \{0,1,\dots, n\}$.
Since $G$ and $C_n^j$ are 2-regular,
they are also $\NH_{\eta}$-cospectral.
By Theorem~\ref{21},
$U_{\theta}(G)$ and $U_{\theta}(C_n^j)$ have the same eigenvalues.
From Lemma~\ref{13},
periodicity is determined by the eigenvalues of the time evolution matrices,
so it is sufficient to discuss only periodicity of $C_n^j$ for $j \in \{0,1,\dots, n\}$.
Let $U_{\theta} = U_{\theta}(C_n^j)$,
and let $a$ be an arbitrary  arc of $C_n^j$.
By Lemma~\ref{25}, we have
\begin{equation} \label{26}
U_{\theta}^{n} \BM{e^{(a)}} = e^{\pm j \eta i} \BM{e^{(a)}}.
\end{equation}
If $\eta \not\in \MB{Q}\pi$,
then $j l \eta \not\in 2\pi \MB{Z}$ for any $l \in \MB{N}$,
so the mixed graph is not periodic.
On the other hand,
we suppose $\eta \in \MB{Q}\pi$, say $\eta = \frac{p}{q}\pi$.
Then,
\[ U_{\theta}^{2qn} \BM{e^{(a)}} = e^{\pm 2qj \cdot \frac{p}{q} \pi i} \BM{e^{(a)}}
= e^{\pm 2pj \pi i} \BM{e^{(a)}} = \BM{e^{(a)}}. \]
Thus, the mixed graph is periodic.

Next, we determine the period $\tau$.
Let $\eta \in \MB{Q}\pi$ and let $\eta = \frac{p}{q}\pi$, where $p$ and $q$ are coprime.
By Lemma~\ref{25},
the vector $U_{\theta}^k \BM{e^{(a)}}$ coincides with a complex multiple of $\BM{e^{(a)}}$
if and only if $k$ is a multiple of $n$.
Thus, the period $\tau$ is a multiple of $n$,
namely, $\tau = ln$ for some $l \in \MB{N}$. 
Then,
\[ \BM{e^{(a)}} = U_{\theta}^{ln} \BM{e^{(a)}} = e^{\pm l j \cdot \frac{p}{q} \pi i} \BM{e^{(a)}}, \]
so $\frac{pjl}{q} \pi \in 2\pi\MB{Z}$, i.e., $pjl \in 2q \MB{Z}$.
We would like to find $\min \{ l \in \MB{N} \mid pjl \in 2q \MB{Z} \}$.
If $j = 0$, we have $\min \{ l \in \MB{N} \mid pjl \in 2q \MB{Z} \} = 1$,
so the period is $n$.
This satisfies (\ref{65}).
We assume that $j > 0$ in the discussion below.

First, we consider the case where $p$ is odd.
We will show that
\[ \min \{ l \in \MB{N} \mid pjl \in 2q \MB{Z} \} = \frac{2q}{(j,2q)}. \]
Since $p$ is odd and $(p,q) = 1$, we have
\begin{equation} \label{27}
(p,2q) = 1.
\end{equation}
Let $d = (j,2q)$.
There exists $j' \in \MB{N}$ such that
\begin{equation} \label{28}
j = j' d
\end{equation}
and
\begin{equation} \label{29}
(j', 2q) = 1.
\end{equation}
Therefore, by (\ref{27}), (\ref{28}) and (\ref{29}), we have
\begin{align*}
\min \{ l \in \MB{N} \mid pjl \in 2q \MB{Z} \}
&= \min \{ l \in \MB{N} \mid jl \in 2q \MB{Z}\} \\
&= \min \{ l \in \MB{N} \mid j' d l \in 2q \MB{Z}\}  \\
&= \min \{ l \in \MB{N} \mid d l \in 2q \MB{Z} \} \\
&= \frac{2q}{d}.
\end{align*}
The last equality is obtained from the fact that $d$ is a divisor of $2q$.

Next, we consider the case where $p$ is even.
We will show that
\begin{equation}
\min \{ l \in \MB{N} \mid pjl \in 2q \MB{Z} \} = \frac{q}{(j,q)}. \label{66}
\end{equation}
If $p=0$,
then $q=1$ since $p$ and $q$ are coprime.
We have $\min \{ l \in \MB{N} \mid pjl \in 2q \MB{Z} \} = 1$ and $\frac{q}{(j,q)} = 1$.
Equality~(\ref{66}) is satisfied in this case.
We assume that $p > 0$.
Since $p$ is even, we have
\begin{equation} \label{31}
p = 2p'
\end{equation}
for some $p' \in \MB{N}$.
Since $p$ and $q$ are coprime,
\begin{equation} \label{32}
(p',q) = 1.
\end{equation}
Let $d = (j,q)$.
There exists $j' \in \MB{N}$ such that
\begin{equation} \label{33}
j = j' d
\end{equation}
and
\begin{equation} \label{34}
(j', q) = 1.
\end{equation}
Therefore, by (\ref{31}), (\ref{32}), (\ref{33}) and (\ref{34}) we have 
\begin{align*}
\min \{ l \in \MB{N} \mid pjl \in 2q \MB{Z} \}
&= \min \{ l \in \MB{N} \mid p' j l \in q \MB{Z}\} \\
&= \min \{ l \in \MB{N} \mid jl \in q \MB{Z}\} \\
&= \min \{ l \in \MB{N} \mid j' d l \in q \MB{Z}\} \\
&= \min \{ l \in \MB{N} \mid d l \in q \MB{Z} \} \\
&= \frac{q}{d}.
\end{align*}
We have the statement.
\end{proof}

\section*{Acknowledgements}
This paper is based on the graduation theses of the second and third authors.
We would like to thank their advisor, Professor Norio Konno,
for his fruitful comments and helpful advice.
The first author is supported by JSPS KAKENHI (Grant No. 20J01175).

\end{document}